\documentclass[11pt, reqno]{amsart}
\usepackage{amsmath}
\usepackage{amsthm}
\usepackage{amssymb}
\usepackage{amscd}
\usepackage{amsfonts}
\usepackage{amsbsy}
\usepackage{epsfig}
\usepackage{psfrag}

\usepackage{a4wide}
\usepackage{amssymb,amsmath,amsthm,mathrsfs}
\usepackage{graphicx}
\usepackage{float}
\usepackage{color, colortbl}
\usepackage{enumerate}
\usepackage{upref}
\usepackage[bookmarks=false]{hyperref}

\theoremstyle{plain}
\newtheorem{theorem}{Theorem}
\newtheorem{lemma}{Lemma}[section]
\newtheorem{corollaryP}[lemma]{Corollary}
\newtheorem{propositionP}[lemma]{Proposition}
\newtheorem{proposition}{Proposition}
\newtheorem{corollary}[theorem]{Corollary}

\theoremstyle{definition}
\newtheorem{definition}{Definition}

\newtheorem*{condition}{Condition}

\theoremstyle{remark}
\newtheorem{remark}{Remark}

\usepackage{enumitem}
\makeatletter
\def\namedlabel#1#2{\begingroup
   #2%
 \def\@currentlabel{#2}%
   \phantomsection\label{#1}\endgroup
}

\def\S{Section}
\def\ie{{\em i.e.,} }

\newfont\bbf{msbm10 at 12pt}
\def\eps{\varepsilon}
\def\R{{\mathbb R}}

\def\N{{\mathbb N}}
\def\Z{{\mathbb Z}}

\def\B{{\mathcal B}}
\def\I{{\mathcal I}}
\def\P{{\mathcal P}}

\def\D{{\mathcal D}}

\def\sm{\setminus}

\def\cv{\ensuremath{\text {Cor}}}
\newcommand{\dif}{\mathrm{d}}

\DeclareMathOperator*{\esssup}{ess\;sup}
\DeclareMathOperator*{\essinf}{ess\;inf}

\def\R{\ensuremath{\mathbb R}}
\def\N{\ensuremath{\mathbb N}}

\def\Z{\ensuremath{\mathbb Z}}
\def\I{\ensuremath{{\bf 1}}}
\def\e{\ensuremath{\text e}}

\def\S{\ensuremath{\mathcal S}}
\def\RR{\ensuremath{\mathcal R}}

\def\B{\ensuremath{\mathcal B}}
\def\l{\ensuremath{\text{Leb}}}

\def\BC{\ensuremath{\mathcal BC}}
\def\M{\ensuremath{\mathcal M}}

\def\P{\ensuremath{\mathcal P}}
\def\p{\ensuremath{\mathbb P}}

\def\QQ{\ensuremath{\mathcal Q}}
\def\HH{\ensuremath{\mathscr H}}

\def\nn{\ensuremath{\mathscr N}}
\def\n{\ensuremath{n}}

\def\X{\mathcal{X}}
\def\ie{{\em i.e.}, }

\def\E{\mathbb E}

\numberwithin{equation}{section}

\def\dist{\ensuremath{\mbox{dist}}}

\def\tag#1{\hfill \qquad  #1}
\def\dist{\mbox{dist}}

\def\le{\leqslant}
\def\ge{\geqslant}

\newcommand{\st}{such that }

\def\Law{{\hat H}}

\begin{document}

\title[The compound Poisson distribution in non-uniformly hyperbolic systems]{The compound Poisson limit ruling periodic extreme behaviour of non-uniformly hyperbolic dynamics}

\author[A. C. M. Freitas]{Ana Cristina Moreira Freitas}
\address{Ana Cristina Moreira Freitas\\ Centro de Matem\'{a}tica \&
Faculdade de Economia da Universidade do Porto\\ Rua Dr. Roberto Frias \\
4200-464 Porto\\ Portugal} \email{amoreira@fep.up.pt}
\urladdr{http://www.fep.up.pt/docentes/amoreira}

\author[J. M. Freitas]{Jorge Milhazes Freitas}
\address{Jorge Milhazes Freitas\\ Centro de Matem\'{a}tica \& Faculdade de Ci\^encias da Universidade do Porto\\ Rua do
Campo Alegre 687\\ 4169-007 Porto\\ Portugal}
\email{jmfreita@fc.up.pt}
\urladdr{http://www.fc.up.pt/pessoas/jmfreita}

\author[M. Todd]{Mike Todd}
\address{Mike Todd\\ Mathematical Institute\\
University of St Andrews\\
North Haugh\\
St Andrews\\
KY16 9SS\\
Scotland} \email{mjt20@st-andrews.ac.uk }
\urladdr{http://www.mcs.st-and.ac.uk/~miket/}

\begin{abstract}
We prove that the distributional limit of the normalised number of returns to small neighbourhoods of periodic points of certain non-uniformly hyperbolic dynamical systems is compound Poisson.
The returns to small balls around a fixed point in the phase space correspond to the occurrence of rare events, or exceedances of high thresholds, so that there is a connection between the laws of Return Times Statistics and Extreme Value Laws. The fact that the fixed point in the phase space is a repelling periodic point implies that there is a tendency for the exceedances to appear in clusters whose average sizes is given by the Extremal Index, which depends on the expansion of the system at the periodic point. 

We recall that for generic points, the exceedances, in the limit, are singular and occur at Poisson times. However, around periodic points, the picture is different: the respective point processes of exceedances converge to a compound Poisson process, so instead of single exceedances, we have entire clusters of exceedances occurring at Poisson times with a geometric distribution ruling  its multiplicity.

The systems to which our results apply include: general piecewise expanding maps of the interval (Rychlik maps), maps with indifferent fixed points (Manneville-Pomeau maps) and Benedicks-Carleson quadratic maps. 

\end{abstract}

\thanks{ACMF was partially supported by FCT grant SFRH/BPD/66174/2009. JMF was partially supported by FCT grant SFRH/BPD/66040/2009. MT was partially supported by NSF grants DMS 0606343 and DMS 0908093. All three authors are supported by FCT (Portugal) projects PTDC/MAT/099493/2008 and PTDC/MAT/120346/2010, which are financed by national and European Community structural funds through the programs  FEDER and COMPETE . All three authors were also supported by CMUP, which is financed by FCT (Portugal) through the programs POCTI and POSI, with national
and European Community structural funds.}

\date{\today}

\maketitle

\section{Introduction}\label{sec:introduction}

The study of extreme events is useful for risk assessment and advanced planning. In many situations, like weather forecasting and the Lorenz equations, natural phenomena can be modelled by dynamical systems. Hence, the study of extreme value laws for data arising from such chaotic dynamical systems has recently received attention from dynamicists interested in understanding better the statistical behaviour of the systems \cite{C01, FF08, FF08a, FFT10, FFT11, GHN11,HNT12}, and also from physicists and meteorologists for whom the estimation of extreme behaviour is of crucial importance \cite{VHF09, HVR12,FLT11,FLT11a,FLT11b,LFW12}.  

We are particularly interested in the convergence of point processes counting the occurrence of extreme events for systems revealing periodic behaviour. This periodicity is responsible for the appearance of clusters of extreme observations (exceedances of high thresholds), which leads to a compound Poisson process, in the limit. The latter can be thought of as having two components: one is the underlying asymptotic Poisson process governing the positions of the clusters of exceedances; and the other is the multiplicity distribution associated to each such Poisson event, which is determined by the average cluster size.  

There are two main approaches to the study of laws of rare events for dynamical systems. One is the study of Extreme Value Laws (EVL) which in the dynamical systems realm is quite recent. The other is the study of Hitting Times Statistics (HTS) or Return Times Statistics (RTS), \ie the limit laws for the normalised waiting times before hitting/returning to  asymptotically small sets, which goes back to \cite{GS90, P91,H93,HSV99}, for example.
In \cite{FFT10}, we showed, under general conditions, the equivalence between these two approaches so that exceedances correspond to hits to small sets.

Almost all papers on the subject deal with hitting/return times to small sets around generic points. The exceptions are \cite{H93} by Hirata, which establishes the distribution of the first return time to small sets around periodic points, for Axiom A systems, and the deep paper \cite{HV09} by Haydn and Vaienti, where the convergence of the normalised number of returns to small sets around periodic points to the compound Poisson distribution was proved for $\phi$ mixing systems. In both papers, the small sets around periodic points considered are dynamically defined cylinders sets. Very recently, in \cite{FFT12}, using an EVL approach, we managed to study the first return to small sets around periodic points for piecewise expanding dynamical systems, but this time the role of the small sets was played by topological balls instead of cylinders, which makes it a stronger result. However, we emphasise that this was only done for the distribution of the first hitting/return time. We also note the work of Chazottes, Coelho and Collet \cite{CCC09}, where the compound Poisson limit  was obtained for the successive closer and closer approximations to subsystems  of finite type (which could be chosen to emulate periodic points) in symbolic dynamics, as well as the work of Ferguson and Pollicott \cite{FP12} which, among other results, improved on \cite{H93}.

In this paper we consider much more general systems. In fact, we prove that for an important and well-studied class of  non-uniformly hyperbolic interval maps the relevant probabilistic law for the normalised number of returns to small neighbourhoods of periodic points is a compound Poisson distribution.  This essentially breaks down into three parts: 

\begin{list}{$\bullet$}
{ \itemsep 0.5mm \topsep 0.0mm \leftmargin=6mm}
\item We give general conditions on the returns which will guarantee a compound Poisson process.  This is expressed in the language of random variables for some process.  We then prove that these conditions  are met by the process consisting of return time statistics to asymptotically small balls around periodic points for uniformly hyperbolic interval maps.

\item We use the approach given in \cite{BSTV03}, but for periodic points.  This essentially says that the Poisson statistics of returns to a periodic point for a first return map are the same as those for the original map.  We then apply these results to maps such as Manneville-Pomeau, which itself is non-uniformly hyperbolic, but which has a uniformly hyperbolic first return map. 

\item We prove the same result, but now for a large set of quadratic maps of the interval, whose first return map is not uniformly hyperbolic.  This requires two tools: the Hofbauer extension, which gives uniformly hyperbolic induced maps, and was employed in this context (but not to periodic points) in \cite{BV03, BT09}; and a Benedicks-Carleson \cite{BC85} type of parameter exclusion argument, which is required here to ensure that the density of our measure doesn't blow up at the periodic point we are concerned with.

\end{list}

\textbf{Notation:} For quantities $(a_r)_r$ and $(b_r)_r$, we write $a_r\asymp b_r$ if there exists $C>0$ such that $\frac1c\le \frac{a_r}{b_r}\le C$ for  all $r$ close enough to its limit (this depends on the context - either the limit is 0 or $\infty$).  Similarly we write $a_r\sim b_r$ if $\lim_r\frac{a_r}{b_r}=1$.

\subsection{Rare events point process, extremes and hitting times}
\label{subsec:point-processes}
The starting point is a stationary stochastic process $X_0, X_1, \ldots$ Even though our results regarding the limit of point processes generated by such stochastic processes apply beyond the dynamical systems realm, since our main application is to dynamical systems, and since it is clearer to present our results in that context, we restrict our discussion to that setting.

Hence, take a system  $(\X,\mathcal
B,\p,f)$, where $\X$ is a Riemannian manifold, $\mathcal B$ is the Borel $\sigma$-algebra, $f:\X\to\X$ is a measurable map
and $\p$ an $f$-invariant probability measure.  

Suppose that the time series $X_0, X_1,\ldots$ arises from such a system simply by evaluating a given  observable $\varphi:\X\to\R\cup\{\pm\infty\}$ along the orbits of the system, or in other words, the time evolution given by successive iterations by $f$:
\begin{equation}
\label{eq:def-stat-stoch-proc-DS} X_n=\varphi\circ f^n,\quad \mbox{for
each } n\in {\mathbb N}.
\end{equation}
Clearly, $X_0, X_1,\ldots$ defined in this way is not an independent sequence.  However, $f$-invariance of $\p$ guarantees that this stochastic process is stationary.

We suppose that the r.v. $\varphi:\X\to\R\cup\{\pm\infty\}$
achieves a global maximum at $\zeta\in \X$ (we allow
$\varphi(\zeta)=+\infty$). 
 We also assume throughout that $\zeta\in \X$ is a repelling periodic point, of prime period\footnote{i.e.,  the \emph{smallest} $n\in \N$ such that $f^n(\zeta)=\zeta$.  Clearly $f^{ip}(\zeta)=\zeta$ for any $i\in \N$. }  $p\in\N$.  So when $\varphi$ and $\p$ are sufficiently regular:

\begin{enumerate}

\item[\namedlabel{item:U-ball}{(R1)}] 
for $u$ sufficiently close to $u_F:=\varphi(\zeta)$,  the event 
\begin{equation*}
\label{def:U}
U(u):=\{x\in\X:\; \varphi(x)>u\}=\{X_0>u\}
\end{equation*} corresponds to a topological ball centred at $\zeta$. Moreover, the quantity $\p(U(u))$, as a function of $u$, varies continuously on a neighbourhood of $u_F$.

The periodicity of $\zeta$ implies that for all large $u$, $\{X_0>u\}\cap f^{-p}(\{X_0>u\})\neq\emptyset$ and the fact that the prime period is $p$ implies that $\{X_0>u\}\cap f^{-j}(\{X_0>u\})=\emptyset$ for all $j=1,\ldots,p-1$. 

\item[\namedlabel{item:repeller}{(R2)}]
 the fact that $\zeta$ is repelling means that we have backward contraction implying that there exists $0<\theta<1$ so that $\bigcap_{j=0}^i f^{-jp}(X_0>u)$ is another ball of smaller radius around $\zeta$ with $$\p\left(\bigcap_{j=0}^i f^{-jp}(X_0>u)\right)\sim(1-\theta)^i\p(X_0>u),$$ for all $u$ sufficiently large. 
 
\end{enumerate}

We are interested in studying the extremal behaviour of the stochastic process $X_0, X_1,\ldots$ which is tied with the occurrence of exceedances of high levels $u$. The occurrence of an exceedance at time $j\in\N_0$ means that the event $\{X_j>u\}$ occurs, where $u$ is close to $u_F$. Observe that a realisation of the stochastic process $X_0, X_1,\ldots$ is achieved if we pick, at random and according to the measure $\p$, a point $x\in\X$, compute its orbit and evaluate $\varphi$ along it. Then saying that an exceedance occurs at time $j$ means that the orbit of the point $x$ hits the ball $U(u)$ at time $j$, \ie $f^j(x)\in U(u)$.

For every $A\subset\R$ we define 
\[
\nn_u(A):=\sum_{i\in A\cap\N_0}\I_{X_i>u}.
\]
In the particular case where $A=I=[a,b)$ we simply write 
$\nn_{u,a}^b:=\nn_u([a,b)).$ 

Observe that $\nn_{u,0}^n$ counts the number of exceedances amongst the first $n$ observations of the process $X_0,X_1,\ldots,X_n$ or, in other words, the number of entrances in $U(u)$ up to time $n$. For high levels of $u$, since $\p(U(u))$ has small measure, an entrance in $U(u)$ (or the occurrence of an exceedance) is considered to be a rare event. This counting of occurrences of rare events will allow us to define the so called point processes of rare events. 

One of the goals here is to study the limit of these point processes which, in particular, will give us the behaviour of the partial maxima of $X_0, X_1,\ldots$ and, equivalently, of the existence of Hitting Time Statistics. In fact, for each $n\in\N$, define the partial maximum
\begin{equation*}
M_n=\max\{X_0,\ldots,X_{n-1}\}.
\end{equation*}
and for a set $A\in\B$ define a new r.v., the \emph{first hitting time} to $A$ denoted by $r_A:\X\to\N\cup\{+\infty\}$ where
\begin{equation*}
r_A(x)=\min\left\{j\in\N\cup\{+\infty\}:\; f^j(x)\in A\right\}.
\end{equation*}
Notice that
\begin{equation}
\label{eq:rel-HTS-EVL}
\{\nn_{u,0}^n=0\}=\{M_n\leq u\}=\{r_{U(u)}>n\}
\end{equation}
If, for a normalising sequence of levels $u_n$ such that
\begin{equation}
\label{eq:def-un}
\lim_{n\to\infty}n\p(X_0>u_n)=\tau,
\end{equation}
for some $\tau\geq0$, there exists a non degenerate distribution function (d.f.) H such that
\begin{equation*}
\lim_{n\to\infty}\p(M_n\leq u_n)=\bar H(\tau),
\end{equation*}
where $\bar H(\tau):=1- H(\tau)$ then we say we have an \emph{Extreme Value Law} (EVL) for $M_n$.
If there exists a non degenerate (d.f.) G such that for all $t\ge 0$,
\begin{equation*}
\lim_{u\to u_F}\p\left(r_{U(u)}\leq \frac t{\p(U(u))}\right)=G(t),
\end{equation*}
then we say we have \emph{Hitting Time Statistics} (HTS) $G$ for balls. 
Similarly, we can restrict our observations to $U(u_n)$: if there exists a non degenerate (d.f.) $\tilde G$ such that  for all $t\ge 0$,
\begin{equation*}
\lim_{u\to u_F}\p\left(r_{U(u)}\leq \frac t{\p(U(u))} \ \middle\vert \ U(u)\right)=\tilde G(t),
\end{equation*} 
then we say we have \emph{Return Time Statistics} (RTS) $\tilde G$ for balls. 

The existence of exponential HTS ($G(t)=1-\e^{-t}$) is equivalent to the existence of exponential RTS ($\tilde G(t)=1-\e^{-t}$). In fact, according to the Main Theorem in \cite{HLV05}, a system has HTS $G$ if and only if it has RTS $\tilde G$ and
\begin{equation}
\label{eq:HTS-RTS}
G(t)=\int_0^t(1-\tilde G(s))\,ds.
\end{equation}

In \cite{FFT10}, we showed that the existence of an EVL for $M_n$ was equivalent to the existence of HTS for balls, with $H=G$, which can be guessed from the second equality in \eqref{eq:rel-HTS-EVL} and the fact that by \eqref{eq:def-un} we may write $n\sim\frac{\tau}{\p(U(u_n))}.$

The motivation for using a normalising sequence $u_n$ satisfying \eqref{eq:def-un} comes from the case when $X_0, X_1,\ldots$ are independent and identically distributed (i.i.d.). In this i.i.d. setting, it is clear that $\p(M_n\leq u)= (F(u))^n$, where $F$ is the d.f. of $X_0$, \ie $F(x):=\p(X_0\leq x)$. Hence, condition \eqref{eq:def-un} implies that
\[
\p(M_n\leq u_n)= (1-\p(X_0>u_n))^n\sim\left(1-\frac\tau n\right)^n\to\e^{-\tau},
\]
as $n\to\infty$. Moreover, the reciprocal is also true. Note that in this case $H(\tau)=1-\e^{-\tau}$ is the standard exponential d.f.

On the other hand, the normalising term in the definition of HTS is inspired by Kac's Theorem which states that the expected amount of time you have to wait before you return to $U(u)$ is exactly $\frac1{\p(U(u))}$. 

In order to define a point process that through \eqref{eq:rel-HTS-EVL} captures the essence of an EVL and HTS, we need to re-scale time using the factor $v:=1/\p(X>u)$ given by Kac's Theorem. However, before we give the definition, we need some formalism. Let $\S$ denote the semi-ring of subsets of  $\R_0^+$ whose elements
are intervals of the type $[a,b)$, for $a,b\in\R_0^+$. Let $\RR$
denote the ring generated by $\S$. Recall that for every $J\in\RR$
there are $k\in\N$ and $k$ intervals $I_1,\ldots,I_k\in\S$ such that
$J=\cup_{i=1}^k I_j$. In order to fix notation, let
$a_j,b_j\in\R_0^+$ be such that $I_j=[a_j,b_j)\in\S$. For
$I=[a,b)\in\S$ and $\alpha\in \R$, we denote $\alpha I:=[\alpha
a,\alpha b)$ and $I+\alpha:=[a+\alpha,b+\alpha)$. Similarly, for
$J\in\RR$ define $\alpha J:=\alpha I_1\cup\cdots\cup \alpha I_k$ and
$J+\alpha:=(I_1+\alpha)\cup\cdots\cup (I_k+\alpha)$.

\begin{definition}
We define the \emph{rare event point process} (REPP) by
counting the number of exceedances (or hits to $U(u_n)$) during the (re-scaled) time period $v_nJ\in\RR$, where $J\in\RR$. To be more precise, for every $J\in\RR$, set
\begin{equation}
\label{eq:def-REPP} N_n(J):=\nn_{u_n}(v_nJ)=\sum_{j\in v_nJ\cap\N_0}\I_{X_j>u_n}.
\end{equation}
\end{definition}

Our main result essentially states that, under certain conditions, the REPP just defined converges in distribution to a compound Poisson process $N$ with intensity $\theta$ and a geometric multiplicity d.f. For completeness, we define here what we mean by a compound Poisson process. (See \cite{K86} for more details).

\begin{definition}
\label{def:compound-poisson-process}
Let $T_1, T_2,\ldots$ be  an i.i.d. sequence of random variables with common exponential distribution of mean $1/\theta$. Let  $D_1, D_2, \ldots$ be another i.i.d. sequence of random variables, independent of the previous one, and with d.f. $\pi$. Given these sequences, for $J\in\RR$, set
$$
N(J)=\int \I_J\;d\left(\sum_{i=1}^\infty D_i \delta_{T_1+\ldots+T_i}\right),
$$ 
where $\delta_t$ denotes the Dirac measure at $t>0$.  Whenever we are in this setting, we say that $N$ is a compound Poisson process of intensity $\theta$ and multiplicity d.f. $\pi$.
\end{definition}
\begin{remark}
\label{rem:poisson-process}
In this paper, the multiplicity will always be integer valued which means that $\pi$ is completely defined by the values $\pi_k=\p(D_1=k)$, for every $k\in\N_0$. Note that, if $\pi_1=1$, then $N$ is the standard Poisson process and, for every $t>0$, the random variable $N([0,t))$ has a Poisson distribution of mean $\theta t$. 
\end{remark}

\begin{remark}
\label{rem:compound-poisson}
In fact, in all statements below, $\pi$ is actually a geometric distribution of parameter $\theta\in (0,1]$,  \ie $\pi_k=\theta(1-\theta)^k$, for every $k\in\N_0$. This means that, as in \cite{HV09}, here, the random variable $N([0,t))$ follows a P\'olya-Aeppli distribution, \emph{i.e.}:
$$
\p(N([0,t))=k)=\e^{-\theta t}\sum_{j=1}^k \theta^j(1-\theta)^{k-j}\frac{(\theta t)^j}{j!}\binom{k-1}{j-1},
$$
for all $k\in\N$ and $\p(N([0,t))=0)=\e^{-\theta t}$. 
\end{remark}

\subsection{Conditions for the convergence of REPP in the presence of clustering}

When the r.v.s in the process $X_0, X_1, \ldots$  are independent, the number of exceedances of the level $u_n$ up to time $n$ is Bernoulli distributed with mean $n\p(X_0>u_n)$. Moreover, condition \eqref{eq:def-un} implies that in the limit we get a Poisson distribution for the number of exceedances.

 In fact, even in the dependent case, if some mixing condition $D(u_n)$ holds and in addition an anti clustering condition $D'(u_n)$ also holds, both introduced by Leadbetter in \cite{L73}, one can show that the REPP converges to a standard Poisson process of intensity $1$ (see for example \cite{LR88}). Since the rates of mixing for dynamical systems are usually given by decay of correlations of observables in certain given classes of functions, it turns out that condition $D(u_n)$ is too strong to be checked for chaotic systems whose mixing rates are known only through decay of correlations. For that reason, motivated by Collet's work \cite{C01}, in \cite{FF08a} the authors suggested a condition $D_2(u_n)$ which together with $D'(u_n)$ was enough to prove the existence of an exponential EVL ($\bar H(\tau)=\e^{-\tau}$) for maxima around non-periodic points $\zeta$. Later on, in \cite{FFT10} the authors provided the so called condition $D_3(u_n)$ which together with $D'(u_n)$ was enough to prove convergence of the REPP to a standard Poisson process of intensity $1$ (see \cite[Theorem~5]{FFT10}).  Condition $D_3(u_n)$ is a slight strengthening of $D_2(u_n)$, but both are much weaker than the original $D(u_n)$, and it is easy to show that they follow easily from sufficiently fast decay of correlations (see \cite[Section~2]{FF08a} and \cite[Proofs of Corollary 6 and Theorem 6]{FFT10}). Thus we were able to prove convergence of the REPP for stochastic processes like \eqref{eq:def-stat-stoch-proc-DS} arising from many chaotic dynamical systems. 
 
In the results mentioned above, condition $D'(u_n)$ prevented the existence of clusters of exceedances, which implied for example that the EVL was a standard exponential $\bar H(\tau)=\e^{-\tau}$. However, when $D'(u_n)$ does not hold, clustering of exceedances is responsible for the appearance of a parameter $0<\theta<1$ in the EVL which now is written as $\bar H(\tau)=\e^{-\theta \tau}$. This parameter, $\theta$ is commonly named \emph{Extremal Index} (EI) and can be defined as follows: if for a sequence of levels $(u_n)_{n\in\N}$ satisfying \eqref{eq:def-un} we have that $\lim_{n\to\infty}\p(M_n\leq u_n)=\e^{-\theta \tau}$,  for some $0\leq\theta\leq 1$, then we say that we have an EI $\theta$. When $\theta=1$ we have no clustering and when $\theta<1$ we have clustering which is as strong as $\theta$ is closed to $0$. In fact, $1/\theta$ can be seen as the average cluster size.

In \cite{FFT12},  the authors established a connection between the existence of an EI less than 1 and periodic behaviour. To be more specific, the main result there states that for dynamically defined stochastic processes as in \eqref{eq:def-stat-stoch-proc-DS}, where $\zeta$ is a repelling periodic point, under some conditions on the dependence structure of the process, there is an EVL for $M_n$ with an EI $\theta$ given by the expansion rate at the repelling periodic point $\zeta$, which is $1/(1-\theta)$.  (Note that we wrote the backward contraction rate $(1-\theta)$ in \ref{item:repeller} so that the EI could be easily identified as being $\theta$.) Around periodic points the rapid recurrence creates clusters of exceedances (hits) which makes it easy to check that condition $D'(u_n)$ fails (see \cite[Section~2.1]{FFT12}). This was a serious obstacle since the theory developed up to \cite{FFT12} was based on Collet's important observation that $D'(u_n)$ could be used not only in the usual way as in Leadbetter's approach, but also to compensate the weakening of the original $D(u_n)$, which allowed the application to chaotic systems with sufficiently fast decay of correlations. To overcome this difficulty we considered the annulus 
\begin{equation*}
\label{eq:def-Qp}
Q_p(u):=U(u)\setminus f^{-p}(U(u))=\{X_0>u, \;X_p\leq u\}
\end{equation*} resulting from removing from $U(u)$ the points that were doomed to return after $p$ steps, which form the smaller ball $U(u)\cap f^{-p}(U(u))$. We named the occurrence of $Q_p(u)$ as an \emph{escape} since it corresponds to the realisations that escape the influence of the underlying periodic phenomena and exit the ball $U(u)$ after $p$ iterates. Then we made the crucial observation that  the limit law corresponding to no entrances up to time $n$ into the ball $U(u_n)$ was equal to the limit law corresponding to no entrances into the annulus $Q_p(u_n)$ up to time $n$ (see \cite[Proposition~1]{FFT12}). This meant that, roughly speaking, the role played by the balls $U(u)$ could be replaced by that of the annuli $Q_p(u)$, with the advantage that points in $Q_p(u)$ were no longer destined to return after just $p$ steps. 

Based in this last observation we proposed two conditions on the dependence structure of $X_0, X_1,\ldots$ that we named $D_p(u_n)$ and $D_p'(u_n)$, which imply the existence of an EVL with EI $\theta<1$ around periodic points. These two conditions can be described as being obtained from $D_2(u_n)$ and $D'(u_n)$ by replacing balls by annuli.     

Regarding the REPP, when $\zeta$ is a repelling periodic point, one might think that to study its limit it would be enough to strengthen $D_p(u_n)$ by replacing the role of exceedances in $D_3(u_n)$ by that of escapes and then mimic the argument in the proof of  \cite[Theorem~5]{FFT10}, which states the convergence of the REPP to a standard Poisson process when $\zeta$ is not periodic and $D'(u_n)$ holds. However, a critical step there is the use of a criterion of Kallenberg \cite[Theorem~4.7]{K86} which applies only to simple point processes, without multiple events, which is not the case here. In particular, this means that by mimicking the proof 
of  \cite[Theorem~5]{FFT10} we can only show that the point process corresponding to counting clusters (instead of exceedances) converges to the usual Poisson process with intensity $1$. Hence, to prove the convergence of the REPP to a compound Poisson process, which we will prove to be the case when $\zeta$ is a periodic repeller, we will compute the Laplace transform of the point process directly and study its limit. 

As usual to obtain the desired convergence we need to impose some conditions on the dependence structure of $X_0, X_1, \ldots$. The first condition,  which we will denote by $D_p(u_n)^*$, is a strengthening of $D_p(u_n)$. Since we cannot use the aforementioned Kallenberg's criterion, this strengthening is a bit stronger than adapting $D_3(u_n)$ in the same way we proceeded with $D_2(u_n)$ to obtain $D_p(u_n)$. However, as in the case of these three just mentioned mixing conditions, it can be easily checked for systems with sufficiently fast decay of correlations.
Before we state the new mixing condition $D_p(u_n)^*$, we need to introduce some notation.  This is illustrated in Figure~\ref{fig:notation}.  Note that the pictures are in two dimensions for expository purposes.  The applications presented in this paper are primarily one-dimensional, but higher dimensional examples are also considered, see also \cite[Section 6.2]{FFT10}.

We define the sequence $\left(U^{(\kappa)}(u)\right)_{\kappa\geq0}$ of nested balls centred at $\zeta$ given by: 
\begin{equation*}
\label{def:U-k}
U^{(0)}(u)=U(u)\quad\mbox{and}\quad U^{(\kappa)}(u)=f^{-p}(U^{(\kappa-1)}(u))\cap U(u)\quad\mbox{for all $\kappa\in\N$.}
\end{equation*}

For $i,\kappa,\ell,s\in\N\cup\{0\}$, we define the following events:
\begin{align*}
Q_{p,i}^\kappa(u)&:=f^{-i}\left(U^{(\kappa)}(u)-U^{(\kappa+1)}(u)\right)\\&
\hspace{1.5mm}=\{X_i>u, X_{i+p}>u,\ldots, X_{i+\kappa p}>u,X_{i+(\kappa+1)p}\leq u\},\\
\QQ_{p,s,\ell}^\kappa(u)&:=\bigcap_{i=s}^{s+\ell-1} \left(Q_{p,i}^\kappa(u)\right)^c \qquad
\HH_{p,s,\ell}^\kappa(u):=\bigcap_{i=s}^{s+\ell-1} f^{-i} \left(\big(U^{(\kappa)}(u)\big)^c\right)
\end{align*}

Observe that for each $\kappa$, the set $Q_{p,0}^\kappa(u)$ corresponds to an annulus centred at $\zeta$. Besides, 
\begin{equation}
\label{eq:U-decomposition}
U(u)=\bigcup_{\kappa=0}^\infty Q_{p,0}^\kappa(u),
\end{equation}
which means that the ball centred at $\zeta$ which corresponds to $\{X_0>u\}$ can be decomposed into a sequence of disjoint annuli where $Q_{p,0}^0(u)$ is the most outward ring and the inner ring $Q_{p,0}^{\kappa+1}(u)$ is sent outward by $f^p$ to the ring $Q_{p,0}^\kappa(u)$, i.e., 
\begin{equation}
\label{eq:ring-dynamics}
f^{p}(Q_{p,0}^{\kappa+1}(u))=Q_{p,0}^\kappa(u).
\end{equation} 

\begin{figure}[h]
  \includegraphics[scale=0.8]{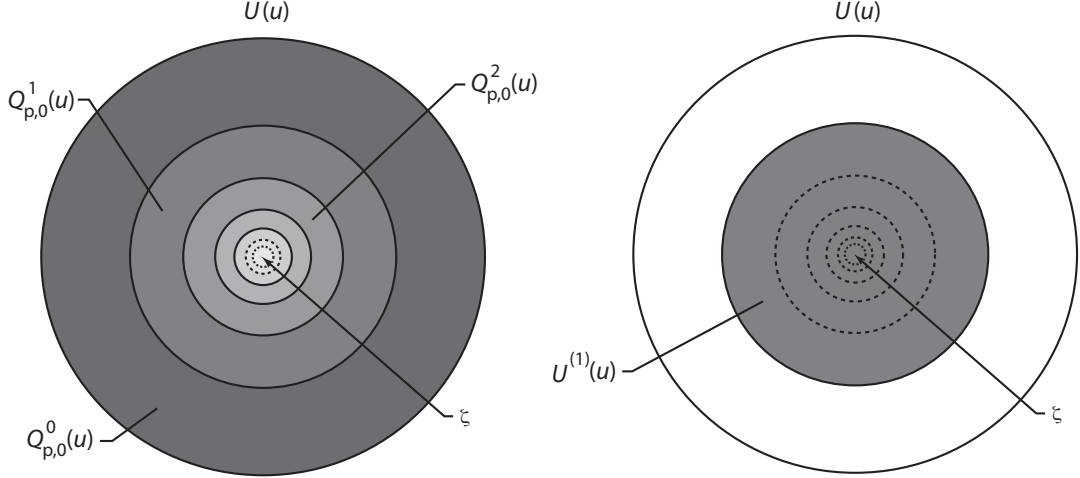} \caption{Notation}\label{fig:notation}
\end{figure}

We are now ready to state:\begin{condition}[$D_p(u_n)^*$]\label{cond:Dp*}We say that $D_p(u_n)^*$ holds
for the sequence $X_0,X_1,X_2,\ldots$ if for any integers $t, \kappa_1,\ldots,\kappa_\varsigma$, $n$ and
 any $J=\cup_{i=2}^\varsigma I_j\in \mathcal R$ with $\inf\{x:x\in J\}\ge t$, 
 \[ \left|\p\left(Q_{p,0}^{\kappa_1}(u_n)\cap \left(\cap_{j=2}^\varsigma \nn_{u_n}(I_j)=\kappa_j \right) \right)-\p\left(Q_{p,0}^{\kappa_1}(u_n)\right)
  \p\left(\cap_{j=2}^\varsigma \nn_{u_n}(I_j)=\kappa_j \right)\right|\leq \gamma(n,t),
\]
where for each $n$ we have that $\gamma(n,t)$ is nonincreasing in $t$  and
$n\gamma(n,t_n)\to0$  as $n\rightarrow\infty$, for some sequence
$t_n=o(n)$. 
\end{condition}
This mixing condition is much weaker than the original $D(u_n)$ from Leadbetter \cite{L73} or $\Delta(u_n)$ from \cite{LR98} because the first of the two events separated by the time gap, namely $Q_{p,0}^{\kappa_1}(u_n)$, is geometrically simple since it corresponds to an annulus. This is not the case for $D(u_n)$ and  $\Delta(u_n)$ which require uniform estimates for events which possibly correspond to geometrically intricate sets. As a consequence of this seemingly small advantage, unlike $D(u_n)$ and  $\Delta(u_n)$, condition $D_p(u_n)^*$ can be easily verified for systems with sufficiently fast decay of correlations.     

Assuming $D_p(u_n)^*$ holds, let $(k_n)_{n\in\N}$ be a sequence of integers such that
\begin{equation}
\label{eq:kn-sequence}
k_n\to\infty\quad \mbox{and}\quad  k_n t_n = o(n). 
\end{equation}

\begin{condition}[$D'_p(u_n)^*$]\label{cond:D'p} We say that $D'_p(u_n)^*$
holds for the sequence $X_0,X_1,X_2,\ldots$ if there exists a sequence $(k_n)_{n\in\N}$ satisfying \eqref{eq:kn-sequence} and such that
\begin{equation}
\label{eq:D'rho-un}
\lim_{n\rightarrow\infty}\,n\sum_{j=1}^{[n/k_n]}\p( Q_{p,0}(u_n)\cap
\{X_j>u_n\})=0.
\end{equation}
\end{condition}
This condition is a  slight strengthening of $D'_p(u_n)$ from \cite{FFT12}, since the occurrence of an escape at time $j$, $Q_{p,j}(u_n)$, was replaced here by the occurrence of the exceedance $\{X_j>u_n\}$. However, since in practice it is easier to check $D'_p(u_n)^*$ and in this paper it makes the forthcoming computations much simpler, we decided to require this stronger version of $D'_p(u_n)$. We recall that $D'_p(u_n)$ is very similar to Leadbetter's $D'(u_n)$ from \cite{L83}, except that instead of preventing the clustering of exceedances it prevents the clustering of escapes by requiring that they should appear scattered fairly evenly through the time interval from $0$ to $n-1$.

We can now state the main theorem. 

\begin{theorem}
\label{thm:convergence-point-process} Let $X_0, X_1, \ldots$ be given by \eqref{eq:def-stat-stoch-proc-DS}, where $\varphi$ achieves a global maximum at the repelling periodic point $\zeta$, of prime period $p$, and conditions \ref{item:U-ball} 
 and \ref{item:repeller} hold. 
 Let
$(u_n)_{n\in\N}$ be a sequence satisfying \eqref{eq:def-un}.  
Assume that conditions 
$D_p(u_n)^*$, $D'_p(u_n)^*$ hold.
Then the EPP $N_n$ converges in distribution to a compound Poisson process $N$ with intensity $\theta$ and multiplicity d.f. $\pi$ given by
$
\pi(\kappa)=\theta(1-\theta)^\kappa,
$
for every $\kappa\in\N_0$, where the extremal index $\theta$ is given by the expansion rate at $\zeta$ stated in \ref{item:repeller}.
\end{theorem}

\begin{remark}
\label{rem:cluster}
The underlying periodicity of the process $X_0, X_1, \ldots$, resulting from the fact that $\varphi$ achieves a global maximum at the periodic point $\zeta$, leads to the appearance of clusters of exceedances whose size depends on the severity of the first exceedance that  begins the cluster. To be more precise, let $x\in \X$: if we have a first exceedance at time $i\in\N$,  which means that $f^i(x)$ enters the ball $U(u)$, then by \eqref{eq:U-decomposition} we must have that $f^i(x)\in Q_{p,0}^\kappa(u)$ for some $\kappa\geq0$, which we express by saying that the entrance at time $i$ had a \emph{depth} $\kappa$. Notice that the deeper the entrance, the closer $f^i(x)$ got to $\zeta$ and the more severe is the exceedance. Now, observe  that if $f^i(x)\in Q_{p,0}^\kappa(u)$ we must have $f^{i+p}(x)\in Q_{p,0}^{\kappa-1}(u),\ldots,f^{i+\kappa p}(x)\in Q_{p,0}^{0}(u)$ and $f^{i+(\kappa+1) p}(x)\notin U(u)$ which means that the size of the cluster initiated at time $i$ is exactly $\kappa+1$ and ends with a visit to the outermost ring $Q_{p,0}^{0}(u)$, which plays the role of an escaping exit from $U(u)$.  So the depth of the entrance in $U(u)$ determines the size of the cluster, and the deeper the entrance, the more severe is the corresponding exceedance and the longer the cluster.     
\end{remark}

\subsection{Applications to dynamical systems}

The theory of HTS is best understood in the context of uniformly hyperbolic dynamical systems.  In the context of HTS to balls, this is usually further restricted to the setting of piecewise conformal systems, in particular to smooth uniformly expanding interval maps. Our first application of the general theorems given above are to such systems, although we do allow quite a lot of flexibility: (countably) infinitely many branches. Our basic assumption will be decay of correlations against $L^1$ observables. Hence we define:
\begin{definition}[Decay of correlations]
\label{def:dc}
Let \( \mathcal C_{1}, \mathcal C_{2} \) denote Banach spaces of real valued measurable functions defined on \( \X \).
We denote the \emph{correlation} of non-zero functions $\phi\in \mathcal C_{1}$ and  \( \psi\in \mathcal C_{2} \) w.r.t.\ a measure $\p$ as
\[
\cv_\p(\phi,\psi,n):=\frac{1}{\|\phi\|_{\mathcal C_{1}}\|\psi\|_{\mathcal C_{2}}}
\left|\int \phi\, (\psi\circ T^n)\, \dif\p-\int  \phi\, \dif\p\int
\psi\, \dif\p\right|.
\]

We say that we have \emph{decay
of correlations}, w.r.t.\ the measure $\p$, for observables in $\mathcal C_1$ \emph{against}
observables in $\mathcal C_2$ if, for every $\phi\in\mathcal C_1$ and every
$\psi\in\mathcal C_2$ we have
 $$\cv_\p(\phi,\psi,n)\to 0,\quad\text{ as $n\to\infty$.}$$
  \end{definition}

We say that we have \emph{decay of correlations against $L^1$
observables} whenever  this holds for $\mathcal C_2=L^1(\p)$  and
$\|\psi\|_{\mathcal C_{2}}=\|\psi\|_1=\int |\psi|\,\dif\p$. 

We state an abstract result that will allow us to show the convergence of the REPP to a compound Poisson process with geometric multiplicity distribution, around repelling periodic points, for systems with decay of correlations against $L^1$ observables. As a corollary we will obtain that this convergence holds for multi dimensional uniformly expanding systems, piecewise expanding systems of the interval (like Rychlik maps) and piecewise expanding systems in higher dimensions like the ones studied by Saussol in \cite{S00}.

\begin{theorem}  \label{thm:decay-L1}
Consider a dynamical system $(\X,\mathcal
B,\p,f)$ for which there exists 
a Banach space $\mathcal C$ of real valued functions such that 
for all $\phi\in\mathcal C$ and $\psi\in L^1(\p)$,
\begin{equation}
\label{DC:L1}
\cv_\mu(\phi,\,\psi,n)\leq  C n^{-2},
\end{equation} where $C>0$ is a constant independent of both $\phi, \psi$.   Let $X_0, X_1, \ldots$ be given by \eqref{eq:def-stat-stoch-proc-DS}, where $\varphi$ achieves a global maximum at the repelling periodic point $\zeta$, of prime period $p$, and conditions \ref{item:U-ball}  and \ref{item:repeller} hold. 
 Let
$(u_n)_{n\in\N}$ be a sequence satisfying \eqref{eq:def-un}.  
 If there exists $C'>0$ such that for all $n$ and $\kappa_1\in\N_0$ we have $\I_{Q_{p,0}^{\kappa_1}(u_n)}\in \mathcal C$, $\|\I_{Q_{p,0}^{\kappa_1}(u_n)}\|_{\mathcal C}\leq C'$ then conditions $D_p(u_n)^*$ and $D_p'(u_n)^*$ hold for $X_0,X_1,\ldots$. 
\end{theorem}

Observe that decay of correlations as in \eqref{DC:L1}, against $L^1(\mu)$ observables, is a very strong property. In fact, regardless of the rate (in this case $n^{-2}$), as long as it is summable, one can  actually  show that the system has exponential decay of correlations of H\"older observables against $L^\infty(\mu)$, as it was shown in \cite[Theorem~B]{AFLV11}. However, this property has been proved for uniformly expanding and piecewise expanding systems.

In particular, we apply our results to a class of Rychlik systems $(Y, f, \phi)$, with equilibrium state $\mu_\phi$, so this measure takes the place of $\p$ in this setting.   For more details see Section~\ref{ssec:Rychlik} and references therein. As a consequence of Theorems~\ref{thm:convergence-point-process} and \ref{thm:decay-L1}, we obtain

\begin{corollary}
Suppose that $(Y,f, \phi)$ is a Rychlik system with equilibrium state $\mu_\phi$.  Then for a periodic point $\zeta$ of prime period $p$, the EPP $N_n$ converges in distribution to a compound Poisson process $N$ with intensity $\theta=1-e^{S_p\phi(\zeta)}$ and multiplicity d.f. $\pi$ given by
$
\pi_\kappa=\theta(1-\theta)^\kappa,
$
for every $\kappa\in\N_0$.
\label{cor:Rychlik poiss}
\end{corollary}

Note that Corollary~\ref{cor:Rychlik poiss} applies to many different uniformly hyperbolic interval maps with `good' invariant measures.  For example it applies to topologically transitive uniformly expanding maps of the interval with an absolutely continuous invariant probability measure (acip), including the systems studied in \cite{BSTV03} and \cite{FP12}.

We can also apply our results to higher dimensional piecewise expanding systems like the ones studied in \cite{S00}. 
\begin{corollary}
Let $f:\X\to\X$ be a piecewise expanding map as defined in \cite[Section~2]{S00}, with an acip $\mu$.  Then for a periodic point $\zeta$ of prime period $p$, the EPP $N_n$ converges in distribution to a compound Poisson process $N$ with intensity $\theta=1-\left|\det D(f^{-p})(\zeta)\right|$ and multiplicity d.f. $\pi$ given by
$
\pi_\kappa=\theta(1-\theta)^\kappa,
$
for every $\kappa\in\N_0$.
\label{cor:saussol}
\end{corollary}

The final part of this paper is concerned with applying the above theory to non-uniformly hyperbolic interval maps.  The main problem here is that in many of these situations it is hard to check conditions $D_p(u_n)^*$ and $D_p'(u_n)^*$, and in some of them it is unlikely that they hold. One way to nevertheless obtain results about these systems is to apply the approach in \cite{BSTV03}, where it was shown that, in many cases, first return maps have the same HTS as the original system for almost every point in the space.  What is different in our case is that we must pick a particular point $\zeta$ and prove the analogous result.  This is challenging because we cannot, as in the proof of \cite[Theorem 1]{BSTV03} simply exclude a zero measure set of points which have bad behaviour: we have to prove our theorem about the particular choice $\zeta$.  On the other hand the proof is assisted by the fact that periodic points have very well-understood behaviour, in particular we can transfer information from small scales to large scales.

We introduce standard measure notation:

\textbf{Notation:}  Given a finite measure $\mu$ on $\X$ and a measurable set $A\subset \X$, let $\mu_A$ be the corresponding conditional measure on $A$, i.e., for $B\subset \X$, $\mu_A(B)=\mu(A\cap B)/\mu(A)$.

As before, let us assume we have a system $f:\X\to \X$, where $f$ is a Borel measurable transformation of the smooth Riemannian manifold $\X$, and now with an invariant probability measure $\mu$ (which we can also denote by $\p$ as usual).  Suppose that $\zeta\in \X$ is a periodic point of prime period $p$.  We pick some subset $\hat \X\subset \X$ and let $R_{\hat \X}$ be the first return time to $\hat \X$, and $\hat f=f^{R_{\hat \X}}$ be the first return map to $\hat \X$.   We will always assume that $\hat \X$ is so small that $R_{\hat \X}(\zeta)= p$.  Also let $\hat\mu=\mu_{\hat \X}$ (alternatively we can write $\p(\cdot |\hat \X)$).  Note that by Kac's Lemma, $\hat\mu$ is $\hat f$-invariant. 

This new setting gives rise to a new set of random variables
$$\hat X_n=\phi\circ \hat f^n.$$
 We can thus consider $\hat \nn_u(\hat v J)$ for $J\in \RR$ and $\hat v=1/\p(U(u)|\hat \X)$ defined analogously to \eqref{eq:def-REPP} for the original system.   

Let $\Law$ be such that, for every $J\in\mathcal S$  and $\kappa\in\N_0$,  
\begin{equation}
\label{eq:def-hatG}
\lim_{u\to u_F} \hat\mu\left(\left\{\hat\nn_u(\hat v J)\leq \kappa\right\}\right)=\Law(J,\kappa),
\end{equation}
where $\Law(J,\cdot)$ corresponds to a d.f. of an integer valued r.v. We will assume that $\Law$ is continuous, in the sense that  $\lim_{\delta\to0}\Law((1\pm\delta)J,\kappa)=\Law(J,\kappa)$, for every $\kappa$. 
Note that we will apply our results to the case that $\Law(J, \kappa)=\p(N(J)=\kappa)$, where $N$ is a compound Poisson process of intensity $\theta$ and a geometric multiplicity. 

In the following theorem we will impose two pairs of conditions on our system,  see Section~\ref{sec:BSTV} for details.  The first pair (M1) and (M2) concern the measure we put on our system: essentially we want it to be an equilibrium state which behaves very like the corresponding conformal measure for some potential.  For example the conformal measure could be Lebesgue measure and the equilibrium state given by any density which is uniformly bounded away from zero and infinity.  The second pair of conditions (S1) and (S2) ensure that the measures `scale well' around our point $\zeta$.  So to continue the example, the dynamics could be a $C^2$ interval/circle map with $\zeta$ a repelling fixed point, so the Lebesgue measure of iterates of a ball centred at $\zeta$ scales like the derivative of the map at $\zeta$.

\begin{theorem}
Suppose that $(\X,f)$ is a dynamical system that $\zeta$ is a periodic point of prime period $p$ and $(\hat \X, \hat f)$ are defined as above.  If the induced system satisfies conditions (M1), (M2), (S1),(S2) and the limit defining $\Law$ as in \eqref{eq:def-hatG} exists and defines a continuous function, then for each $J\in \RR$ there exists $(\delta_{J}(u))_{u>0}> 0$ such that $\delta_{J}(u)\searrow 0$ as $u\nearrow u_F$  and for $\kappa\in\N_0$,
\begin{equation*}
\left| \p\left(\{\nn_u( v J)\le \kappa\}\right)-\Law(J, t) \right|<\delta_{J}(r).\label{eq:N_u orig}
\end{equation*}
\label{thm:ind same}
\end{theorem}

Since, as shown in \cite{BSTV03},  for $\alpha\in (0,1)$  the Manneville-Pomeau map 
\begin{equation*}
g(x)=\begin{cases} x(1+2^\alpha x^\alpha) & \text{ for } x\in [0, 1/2)\\
2x & \text{ for } x\in [1/2, 1)\end{cases}
\end{equation*}
 with the natural potential $\phi=-\log|Dg|$, has a canonical first return map $\hat g:[1/2, 1)\to [1/2, 1)$ (so $\X=[1/2, 1)$) for which $([1/2, 1), \hat g, \hat \phi)$ is a Rychlik system.  Here $\hat\phi$ is the induced version of the potential $\phi$, which is $\log|D\hat g|$ in this case.  Note that the equilibrium state is the acip $\mu$. Theorems~\ref{cor:Rychlik poiss} and \ref{thm:ind same} imply that for any periodic point $\zeta\neq 0$ for $g$, we have the conclusions of Theorem~\ref{thm:convergence-point-process}.

We next extend the application of our theory to interval maps with a critical point.  We consider a class of $C^3$ unimodal interval maps $f:I \to I$ with an acip.  Let $c$ be the critical point.  Such a map is called \emph{$S$-unimodal} if it has negative \emph{Schwarzian derivative}, i.e., $D^3f(x)/Df(x) - \frac32 (D^2f(x)/Df(x))^2<0$ for any $x\in I\sm\{c\}$.  We say that $c$ is \emph{non-flat} if there exists $\ell\in (1,\infty)$ such that $\lim_{x\to c}|f(x)-f(c)|/|x-c|^\ell$ exists and is positive.  Here $\ell$ is called the \emph{order} of the critical point. 

As in  \cite{BSTV03}, if the critical point  has an orbit which is not dense in $I$ (eg the Misiurewicz case), it is possible to construct a first return map which gives a Rychlik system with the natural potential, and thus  the conclusions of Theorem~\ref{thm:convergence-point-process} hold for the system with its acip.  Our last main result goes beyond this theory since it applies to maps where the critical point has a dense orbit and first return maps are not Rychlik.  We can nevertheless recover our limit theorems using the Hofbauer extension techniques of \cite{BV03}.

We will assume that our maps satisfy the \emph{summability condition}:
\begin{equation}
\sum_{n\ge 1} \frac1{|Df^n(f(c))|^\ell}<\infty.\label{eq:summable}
\end{equation}
Nowicki and van Strien \cite{NS91} showed that under this condition, $f$ has an acip $\mu$.  The support of this measure, the usual metric attractor (see for example \cite[Chapter V.1]{MS93})  is a finite union of intervals.  If we were to assume that $f$ was topologically transitive on $I$, i.e. there exists $x_0\in I$ such that $\overline{\cup_{n\ge 0}f^n(x_0)}=I$, then the support is equal to the whole of $I$.  In any case, the metric entropy $h(\mu)$ is strictly positive.

\begin{theorem}
Suppose that $f:I\to I$ is an $S$-unimodal map with non-flat critical point with order $\ell$ satisfying the summability condition \eqref{eq:summable}.  If $\zeta$ is a repelling periodic point of prime period $p$, in the support of the acip $\mu$ and such that
 \begin{equation}
\sum_{n\ge 1} \frac1{|Df^n(f(c))|^\ell|f^k(f(c))-\zeta|^{1-\frac1\ell}}<\infty,\label{eq:dens}
\end{equation}
 then the EPP $N_n$ converges in distribution to a compound Poisson process $N$ with intensity $\theta=1-\frac1{|Df^p(z)|}$ and multiplicity d.f. $\pi$ given by
$
\pi(\kappa)=\theta(1-\theta)^\kappa,
$
for every $\kappa\in\N_0$.
\label{thm:BrV new}
\end{theorem}

In Section~\ref{sec:BC arg}, we use a parameter exclusion argument to prove that there is a large set of parameters in the family of quadratic maps which satisfy the conditions of Theorem~\ref{thm:BrV new} (specifically \eqref{eq:dens}).

\section{Convergence of the REPP to a compound Poisson process}
\label{sec:REPP}

The goal of this section is to prove Theorem~\ref{thm:convergence-point-process}. The major obstacle we have to deal with is the fact that our condition $D_p(u_n)^*$ is much weaker than the usual conditions $D(u_n)$ and $\Delta(u_n)$ from Leadbetter. We will overcome this difficulty with the help of $D'_p(u_n)^*$ and the special structure borrowed by the underlying periodicity. We already faced the same problem in \cite{FFT12}, where the solution was to observe that replacing the role of exceedances by that of escapes does not alter the limit law. Here, the limit for the point process counting exceedances forces us to take a deeper analysis in order to count the weight of the number of exceedances inside a cluster in the overall sum.  Roughly speaking, we will see that a cluster of size $\kappa$ corresponds to an entrance in $Q^\kappa_{p,0}$ and the measure of these rings will give us in particular the multiplicity d.f.   For this convergence we need the following definitions.

\begin{definition}
\label{def:Laplace-rv}
Let $Z$ be a non-negative, integer valued random variable whose distribution is given by $f_Z(\kappa)=\p(Z=\kappa)$. For every $y\in\R_0^+$, the \emph{Laplace transform} $\phi(y)$ of the distribution $f_Z$ is given by
\[
\phi(y):=\E\left(\e^{-yZ}\right)=\sum_{\kappa=0}^\infty \e^{-y \kappa}f_Z(\kappa).
\]  
\end{definition}

\begin{definition}
\label{def:Laplace-point-process}
For a point process $M$ on $\R_0^+$ and $\varsigma$ intervals $I_1, I_2,\ldots, I_\varsigma\in\S$ and non-negative $y_1, y_2,\ldots,y_\varsigma$, we define the \emph{joint Laplace transform} $\psi(y_1, y_2,\ldots,y_\varsigma)$ by
\[
\psi_M(y_1, y_2,\ldots,y_\varsigma)=\E\left(\e^{-\sum_{j=1}^\varsigma y_jM(I_j)}\right).
\] 
\end{definition}

If $M=N$ is a compound Poisson point process with intensity $\lambda$ and multiplicity distribution $\pi$, then given $\varsigma$ intervals $I_1, I_2,\ldots, I_\varsigma\in\S$ and non-negative $y_1, y_2,\ldots,y_\varsigma$ we have:
\[
\psi_N(y_1, y_2,\ldots,y_\varsigma)=\e^{-\lambda\sum_{\ell=1}^\varsigma (1-\phi(y_\ell))|I_\ell|},
\]
where $\phi(y)=\sum_{\kappa=0}^\infty \e^{-y \kappa}\pi(\kappa)$ is the Laplace transform of the multiplicity distribution. 
 
We begin the proof of Theorem~\ref{thm:convergence-point-process} with a series of abstract Lemmata to capture this correspondence between clusters' size and the depth of the entrances. Then we use their estimates to compute the Laplace transform of the REPP and finally show that it converges to the Laplace transform of a compound Poisson process with the right multiplicity d.f.

The next Lemma is a very important observation which will basically allow us to replace the event corresponding to no entrances in $U^{(\kappa)}(u)$, up to a given time, by the event corresponding to no entrances in $Q^\kappa_{p,0}(u)$. The idea behind it is that \eqref{eq:ring-dynamics} imposes a structure that forces an early entrance in $U^{(\kappa)}(u)$, which does not imply an entrance in $Q^\kappa_{p,0}(u)$ during the considered time frame, to be very deep and consequently very unlikely.        

\begin{lemma}
\label{Lem:disc-ring}
For any $p\in\N$, $s,\kappa\in\N\cup\{0\}$ and $u$ sufficiently close to $u_F=\varphi(\zeta)$ we have
\[
\left|\p\big(\HH_{p,0,s}^\kappa(u)\big)-\p\big(\QQ_{p,0,s}^\kappa(u)\big)\right|\leq p\sum_{i=\kappa+1}^{\infty} \p(U^{(i)}(u)).
\]
\end{lemma}
\begin{proof}
First observe that since $Q_{p,0}^\kappa(u)\subset U^{(\kappa)}(u)$ we have 
$\HH_{p,0,s}^\kappa(u)\subset\QQ_{p,0,s}^\kappa(u).$ 
Next, note that if $\QQ_{p,0,s}^\kappa(u)\setminus \HH_{p,0,s}^\kappa(u)$ occurs, then we may define $i=\min\{j\in\{0,1,\ldots s\}:\; X_j\in U^{(\kappa+1)}(u)\}$ and $\ell_i=[ \frac {s-i}{p}]$. But since $\QQ_{p,0,s}^\kappa(u)$ does occur, we must have $X_{i}\in U^{(\kappa+\ell_i+1)}(u)$, otherwise, by \eqref{eq:ring-dynamics}, there would exist $j_i\leq \ell_{i}$ such that $X_{i+j_ip}\in Q_{p,0}^\kappa (u)$,  which contradicts the occurrence of $\QQ_{p,0,s}^\kappa(u_n)$. This means that
\[
\QQ_{p,0,s}^\kappa(u)\setminus \HH_{p,0,s}^\kappa(u)\subset \bigcup_{i=0}^{s} \{X_{i}\in U^{(\kappa+\ell_i+1)}(u)\}.
\]
Hence, it follows that
\[
\left|\p\big(\HH_{p,0,s}^\kappa\big)-\p\big(\QQ_{p,0,s}^\kappa\big)\right|= \p\left(\QQ_{p,0,s}^\kappa(u)\setminus \HH_{p,0,s}^\kappa(u)\right)\leq p\sum_{i=\kappa+1}^{[s/p]} \p(U^{(i)}(u))\leq p\sum_{i=\kappa+1}^{\infty} \p(U^{(i)}(u)).
\]
\end{proof}
The next result is a technical, but useful, lemma which is a consequence of the law of total probability.

\begin{lemma}
\label{lem:no-entrances-ring}
For any $p\in\N$, $s,\kappa\in\N\cup\{0\}$ and $u$ sufficiently close to $u_F=\varphi(\zeta)$ we have
\[
\left|\p\big(\QQ_{p,0,s}^\kappa(u)\big)-\big(1-s\p(Q_{p,0}^\kappa(u))\big)\right|\leq s\sum_{j=p+1}^s\p(Q_{p,0}^0(u)\cap \{X_j>u\})
\]
\end{lemma}

\begin{proof}
Since $(\QQ_{p,0,s}^\kappa(u))^c=\cup_{i=0}^sQ_{p,i}^\kappa(u)$ it is clear that
\[
\left|1-\p(\QQ_{p,0,s}^\kappa(u))-s\p(Q_{p,0}^\kappa(u))\right|\leq \sum_{i=0}^s\sum_{j=i+p+1}^s\p(Q_{p,i}^\kappa(u)\cap Q_{p,j}^\kappa(u)).
\]
The result now follows by stationarity plus the two following facts: $Q_{p,j}^\kappa(u)\subset \{X_j>u\}$ and the fact that between two entrances to $Q_{p,0}^\kappa(u)$, at times $i$ and $j$, there must have existed an escape at time $i+\kappa p$, \ie the occurrence of  $Q_{p,i+\kappa p}^0(u)$.
\end{proof}

The next result gives an estimate for difference between occurring less than $\kappa$ exceedances, during a certain time interval, and the event corresponding to no entrances in $U^{(\kappa)}(u)$ during that time frame.

\begin{lemma}
\label{lem:entrances-ball-depth}
For any $p\in\N$, $s,\kappa\in\N\cup\{0\}$ and $u$ sufficiently close to $u_F=\varphi(\zeta)$ we have
\[
\left|\p\big(\nn_{u,0}^{s+1}\leq\kappa\big)-\p(\HH_{p,0,s}^\kappa(u)\big)\right|\leq (s-p)\sum_{j=p+1}^s\p(Q_{p,0}^0(u)\cap \{X_j>u\})+\kappa p \p(U^\kappa)
\] 
\end{lemma}
\begin{proof}
We start by observing that
\[
A_{0,s}^\kappa(u):=\left\{\nn_{u,0}^{s+1}\leq\kappa\right\}\cap \left(\HH_{p,0,s}^\kappa(u)\right)^c\subset \bigcup_{i=s-\kappa p}^s \{X_i\in U^{(\kappa)}(u)\}= \bigcup_{i=s-\kappa p}^s f^{-i}(U^{(\kappa)}(u)),
\]
since by Remark~\ref{rem:cluster} an entrance in $U^{(\kappa)}(u)$ leads to a cluster of at least $\kappa+1$ exceedances separated by $p$ units of time. This means that the only way an entrance in $U^{(\kappa)}(u)$ can occur and yet the number of exceedances during the time period from $0$ to $s$ is not greater than $\kappa$, is if the entrance in $U^{(\kappa)}(u)$ happens at a time such that the corresponding cluster extends beyond the time  $s$. So by stationarity,
\[
\p\left(A_{0,s}^\kappa(u)\right)\leq \kappa p \p(U^\kappa(u))
\]
Now, we note that
\[
B_{0,s}^\kappa(u):=\left\{\nn_{u,0}^{s+1}>\kappa\right\}\cap \HH_{p,0,s}^\kappa(u)\subset \bigcup_{i=0}^{s-p}\bigcup_{j>i+p}^s Q_{p,i}^0(u)\cap \{X_j\in U^{(0)}(u)\}.
\]
This is because no entrance in $U^{\kappa}(u)$ during the time period $0,\ldots,s$ implies, by Remark~\ref{rem:cluster}, that the maximum cluster size in that period is at most $\kappa$. Hence, in order to count more than $\kappa$ exceedances, there must be at least two distinct clusters during the time period $0,\ldots,s$. Since each cluster ends with an escape, \ie an entrance in $Q_{p,0}^0(u)$, then this must have happened at some moment $i\in\{0,\ldots,s-p\}$ which was then followed by another exceedance at some subsequent instant $j>i$ where a new cluster is begun. Consequently, by stationarity, we have
\[
\p\left(B_{0,s}^\kappa(u)\right)\leq (s-p)\sum_{j=p+1}^s\p\left(Q_{p,0}^0(u)\cap \{X_j>u\}\right).
\]
The result follows now at once since
\begin{align*}
\left|\p\left(\nn_{u,0}^{s+1}\leq\kappa\big)-\p(\HH_{p,0,s}^\kappa(u)\right)\right|&\leq\p\left(\left\{\nn_{u,0}^{s+1}\leq\kappa\right\}\triangle \HH_{p,0,s}^\kappa(u)\right)
=\p(A_{0,s}^\kappa(u))+\p(B_{0,s}^\kappa(u)).
\end{align*}    
\end{proof}

The lemmas above pave the way for the proof of the next five results, which will then enable us to prove the convergence the Laplace transforms of our point processes to that of a compound Poisson distribution.

\begin{corollaryP}
\label{cor:k-ring}
Assuming that $\varphi$ achieves a global maximum at the repelling periodic point $\zeta$, of prime period $p$, and conditions \ref{item:U-ball} and \ref{item:repeller} hold, there exists $C>0$ depending only on $\theta$ given by property \ref{item:repeller} such that for any $s,\kappa\in\N$ and $u$ sufficiently close to $u_F=\varphi(\zeta)$ we have for $\kappa>0$
$$
\left|\p\big(\nn_{u,0}^{s+1}=\kappa\big)-s\left(\p(Q_{p,0}^{\kappa-1}(u))-\p(Q_{p,0}^\kappa(u))\right)\right|\leq 4s\sum_{j=p+1}^s\p(Q_{p,0}^0(u)\cap \{X_j>u\})+2C\,\p(X_0>u_n),
$$
and in the case $\kappa=0$
$$
\left|\p\big(\nn_{u,0}^{s+1}=0\big)-\left(1-s\p(Q_{p,0}^0(u))\right)\right|\leq 2s\sum_{j=p+1}^s\p(Q_{p,0}^0(u)\cap \{X_j>u\})+C\,\p(X_0>u).
$$
\end{corollaryP}
\begin{proof}
Using Lemmas~\ref{Lem:disc-ring}-\ref{lem:entrances-ball-depth}, recalling that by assumption \ref{item:repeller} about the repelling periodic point, we have $1-\theta<1$ and that for every non-negative integer $\kappa$, $\p(U^{\kappa}(u))\sim(1-\theta)^\kappa \p(U^{(0)}(u))$, it follows that there exists a constant $C>0$ such that for every $\kappa\in\N_0$ we have
\begin{align*}
\big|\p\big(\nn_{u,0}^{s+1}\leq\kappa\big)-&(1-s\p(Q_{p,0}^\kappa(u)))\big|\leq \left|\p\big(\nn_{u,0}^{s+1}\leq\kappa\big)-\p(\HH_{p,0,s}^\kappa(u)\big)\right|\\
& \quad + \left|\p\big(\HH_{p,0,s}^\kappa(u)\big)-\p\big(\QQ_{p,0,s}^\kappa(u)\big)\right|+\left|\p\big(\QQ_{p,0,s}^\kappa(u)\big)-\big(1-s\p(Q_{p,0}^\kappa(u))\big)\right|\\
&\leq2(s-p)\sum_{j=p+1}^s\p(Q_{p,0}^0(u)\cap \{X_j>u\})+\kappa p \p(U^{\kappa})+p\sum_{i=\kappa+1}^{\infty} \p(U^{(i)}(u))\\
& \leq 2s\sum_{j=p+1}^s\p(Q_{p,0}^0(u)\cap \{X_j>u\})+C\,\p(X_0>u).
\end{align*}
Since $\{\nn_{u,0}^{s+1}\leq0\}=\{\nn_{u,0}^{s+1}=0\}$, the result is clear for $\kappa=0$. The case $\kappa\in\N$ follows easily after observing that
$\p\big(\nn_{u,0}^{s+1}=\kappa\big)=\p\big(\nn_{u,0}^{s+1}\leq\kappa\big)-\p\big(\nn_{u,0}^{s+1}\leq\kappa-1\big)$ which implies
\begin{align*}
\left|\p\big(\nn_{u,0}^{s+1}=\kappa\big)-s\left(\p(Q_{p,0}^{\kappa-1}(u))-\p(Q_{p,0}^\kappa(u))\right)\right|&\leq \big|\p\big(\nn_{u,0}^{s+1}\leq\kappa\big)-(1-s\p(Q_{p,0}^\kappa(u))\big|\\&\quad+\big|\p\big(\nn_{u,0}^{s+1}\leq\kappa-1\big)-(1-s\p(Q_{p,0}^{\kappa-1}(u))\big|.
\end{align*} 
\end{proof}

\begin{corollaryP}
\label{cor:exponential}
 Assuming that $\varphi$ achieves a global maximum at the repelling periodic point $\zeta$, of prime period $p$, and conditions \ref{item:U-ball}  and \ref{item:repeller} hold, there exists $C>0$ depending only on $\theta$ given by property \ref{item:repeller} such that for any $s\in\N$, $y\geq0$ and $u$ sufficiently close to $u_F=\varphi(\zeta)$ we have 
\begin{align*}
\Big|\E\left(\e^{-y\nn_{u,0}^{s+1}}\right)-&\big(1-s\p(Q_{p,0}^0(u)\big)\big)-\sum_{\kappa=1}^{\lfloor s/p\rfloor}\e^{-y\kappa}s\left(\p(Q_{p,0}^{\kappa-1}(u))-\p(Q_{p,0}^\kappa(u))\right)\Big| \\
&\hspace{3cm} \leq C\left(s\sum_{j=p+1}^s\p(Q_{p,0}^0(u)\cap \{X_j>u\})+\p(X_0>u)\right).
\end{align*}
\end{corollaryP}
\begin{proof}
Since, up to time $s$ there can be at most $\lfloor s/p\rfloor$ exceedances we have
\[
\E\left(\e^{-y\nn_{u,0}^{s+1}}\right)=\sum_{\kappa=0}^{\lfloor s/p\rfloor} \e^{-y\kappa}\p(\nn_{u,0}^{s+1}=\kappa),
\]
and the result now follows from Corollary~\ref{cor:k-ring} and the fact that $\sum_{\kappa=0}^\infty e^{-y\kappa}<\infty$, for every $y>0$ and from the fact that, for all $\kappa$, we have $e^{-y\kappa}= 1$ for $y= 0$.
\end{proof}

\begin{propositionP}
\label{prop:main-step}
Assume that $\varphi$ achieves a global maximum at the repelling periodic point $\zeta$, of prime period $p$, and conditions \ref{item:U-ball} and \ref{item:repeller} hold. Let $s, t, \varsigma\in\N$ and consider $\kappa_1\in\N_0$, $\underline{\kappa}=(\kappa_2,\ldots,\kappa_\varsigma)\in\N_0^{\varsigma-1}$, $s+t<a_2<b_2<a_3<\ldots<b_\varsigma\in\N_0$. For $u$ sufficiently close to $u_F=\varphi(\zeta)$ we have
\begin{align*}
\big|\p(\nn_{u,0}^{s+1}=\kappa_1, \nn_{u,a_2}^{b_2}&=\kappa_2, \ldots, \nn_{u,a_\varsigma}^{b_\varsigma}=\kappa_\varsigma)-\p(\nn_{u,0}^{s+1}=\kappa_1)\p( \nn_{u,a_2}^{b_2}=\kappa_2, \ldots, \nn_{u,a_\varsigma}^{b_\varsigma}=\kappa_\varsigma)\big|\\
&\hspace{1.4cm} \leq C \left( s\,\iota(u,t)+s\sum_{j=p+1}^s\p(Q_{p,0}^0(u)\cap \{X_j>u\})+\p(U^{(0)}(u))\right).
\end{align*}
for some $C>0$ depending only on on $\theta$ given by property \ref{item:repeller} and where \begin{equation}
\label{eq:def-iota}
\iota(u,t)=\sup_{s\in \N}\max_{i=0,\ldots, s}\left\{\left|\p(Q_{p,i}^{\kappa_1})\p\big(\cap_{j=2}^\varsigma \{\nn_{u,a_j}^{b_j}=\kappa_j\}\big)-\p\big( \cap_{j=2}^\varsigma \{\nn_{u,a_j}^{b_j}=\kappa_j\}\cap Q_{p,i}^{\kappa_1}\big)\right|\right\}.
\end{equation}

\end{propositionP}
\begin{proof}
Let 
\begin{align*}
A_{\kappa_1,\underline{\kappa}}&:=\{\nn_{u,0}^{s+1}\leq\kappa_1, \nn_{u,a_2}^{b_2}=\kappa_2, \ldots, \nn_{u,a_\varsigma}^{b_\varsigma}=\kappa_\varsigma\},\\
\tilde A_{\kappa_1,\underline{\kappa}} &:=\QQ_{p,0,s}^{\kappa_1}(u)\cap\{ \nn_{u,a_2}^{b_2}=\kappa_2, \ldots, \nn_{u,a_\varsigma}^{b_\varsigma}=\kappa_\varsigma\},\\
A^*_{\kappa_1,\underline{\kappa}}&:=\{\nn_{u,0}^{s+1}=\kappa_1, \nn_{u,a_2}^{b_2}=\kappa_2, \ldots, \nn_{u,a_\varsigma}^{b_\varsigma}=\kappa_\varsigma\},\\
B_{\kappa_1}&:=\{\nn_{u,0}^{s+1}\leq\kappa_1\}, \quad \tilde B_{\kappa_1}:=\QQ_{p,0,s}^{\kappa_1}(u),  \quad B^*_{\kappa_1}:=\{\nn_{u,0}^{s+1}=\kappa_1\},\\
D^{\underline{\kappa}}&:=\{\nn_{u,\ell}^{b_2}=\kappa_2, \ldots, \nn_{u,a_\varsigma}^{b_\varsigma}=\kappa_\varsigma\}.
\end{align*}

By Lemmas~\ref{Lem:disc-ring} and \ref{lem:entrances-ball-depth} we have
\begin{align}
\label{eq:approx1}
\left|\p(A_{\kappa_1,\underline{\kappa}})-\p(\tilde A_{\kappa_1,\underline{\kappa}})\right|&\leq \left|\p(B_{\kappa_1})-\p(\tilde B_{\kappa_1})\right|\nonumber\\
&\leq \left|\p(\nn_{u,0}^{s+1}\leq\kappa_1) -\p(\HH_{p,0,s}^{\kappa_1}(u))\right|+\left|\p(\HH_{p,0,s}^{\kappa_1}(u))-\p(\QQ_{p,0,s}^{\kappa_1}(u))\right|\nonumber\\
&\leq s\sum_{j=p+1}^s\p(Q_{p,0}^0(u)\cap \{X_j>u\})+\kappa_1 p \p(U^{\kappa_1})+p\sum_{i=\kappa_1+1}^{\infty} \p(U^{(i)}(u))
\end{align}
Using stationarity and adapting the proof of Lemma~\ref{lem:no-entrances-ring},
 it follows that $\left|\p(\tilde A_{\kappa_1,\underline{\kappa}})-(1-s\p(Q_{p,0}^{\kappa_1}))\p(D^{\underline \kappa})\right|\leq Err,$ where
$$Err=  \left|s\p(Q_{p,0}^{\kappa_1})\p(D^{\underline{\kappa}})-\sum_{i=0}^s\p(Q_{p,i}^{\kappa_1}\cap D^{\underline{\kappa}})\right|+s\sum_{j=p+1}^s\p(Q_{p,0}^0(u)\cap \{X_j>u\}).$$

Now, since, by definition of $\iota(u,t)$, $$ \left|s\p(Q_{p,0}^{\kappa_1})\p(D^{\underline \kappa})-\sum_{i=0}^s\p(Q_{p,i}^{\kappa_1}\cap D^{\underline \kappa})\right|=\left|\sum_{i=0}^s\p(Q_{p,i}^{\kappa_1})\p(D^{\underline \kappa})-\p(Q_{p,i}^{\kappa_1}\cap D^{\underline \kappa})\right|\leq s\iota(u,t),$$
we conclude that 
\begin{equation}
\label{eq:approx2}
\left|\p(\tilde A_{\kappa_1,\underline{\kappa}})-(1-s\p(Q_{p,0}^{\kappa_1}))\p(D^{\underline \kappa})\right|\leq s\iota(u,t)+s\sum_{j=p+1}^s\p(Q_{p,0}^0(u)\cap \{X_j>u\}). 
\end{equation}
Also, by Lemma~\ref{lem:no-entrances-ring} we have
\begin{equation}
\label{eq:approx3}
\left|\p(\tilde B_{\kappa_1})\p(D^{\underline \kappa})-(1-s\p(Q_{p,0}^{\kappa_1}))\p(D^{\underline \kappa})\right|\leq s\sum_{j=p+1}^s\p(Q_{p,0}^0(u)\cap \{X_j>u\}).
\end{equation}
Putting together the estimates \eqref{eq:approx1},\eqref{eq:approx2} and \eqref{eq:approx3} we get
\begin{align*}
|\p( & A_{\kappa_1,\underline{\kappa}})-\p( B_{\kappa_1})\p(D^{\underline \kappa})|\leq \left|\p( A_{\kappa_1,\underline{\kappa}})-\p( \tilde A_{\kappa_1,\underline{\kappa}})\right|+\left|\p(\tilde A_{\kappa_1,\underline{\kappa}})-(1-s\p(Q_{p,0}^{\kappa_1}))\p(D^{\underline \kappa})\right|\\
&\hspace{2cm}+\left|\p(\tilde B_{\kappa_1})\p(D^{\underline \kappa})-(1-s\p(Q_{p,0}^{\kappa_1}))\p(D^{\underline \kappa})\right|+\left|\p(B_{\kappa_1})-\p(\tilde B_{\kappa_1})\right|\p(D^{\underline \kappa})\\
&\leq s\iota(u,t)+4s\sum_{j=p+1}^s\p(Q_{p,0}^0(u)\cap \{X_j>u\})+2\kappa_1 p \p(U^{(\kappa_1)}(u))+2p\sum_{i=\kappa_1+1}^{\infty} \p(U^{(i)}(u))
\end{align*}
Since $\p(A^*_{\kappa_1,\underline{\kappa}})=\p(A_{\kappa_1,\underline{\kappa}})-\p(A_{\kappa_1-1,\underline{\kappa}})$ and $\p(B^*_{\kappa_1})=\p(B_{\kappa_1})-\p(B_{\kappa_1-1})$, we have
\begin{align*}
|\p( & A^*_{\kappa_1,\underline{\kappa}})-\p( B^*_{\kappa_1})\p(D^{\underline \kappa})|\leq |\p( A_{\kappa_1,\underline{\kappa}})-\p( B_{\kappa_1})\p(D^{\underline \kappa})|+|\p( A_{\kappa_1-1,\underline{\kappa}})-\p( B_{\kappa_1-1})\p(D^{\underline \kappa})|\\
&\quad \leq 2s\iota(u,t)+8s\sum_{j=p+1}^s\p(Q_{p,0}^0(u)\cap \{X_j>u\})+4\kappa_1 p \p(U^{(\kappa_1-1)}(u))+4p\sum_{i=\kappa_1}^{\infty} \p(U^{(i)}(u))
\end{align*}
Recalling that by assumption \ref{item:repeller} about the repelling periodic point, we have $1-\theta<1$ and, for every non-negative integer $\kappa$, $\p(U^{\kappa}(u))\sim(1-\theta)^\kappa \p(U^{(0)}(u))$, then there exists a constant $C>0$ such that for all $\kappa_1,\underline{\kappa}$ we have
$$|\p( A^*_{\kappa_1,\underline{\kappa}})-\p( B^*_{\kappa_1})\p(D^{\underline \kappa})|\leq C\left( s\iota(u,t)+s\sum_{j=p+1}^s\p(Q_{p,0}^0(u)\cap \{X_j>u\})+\p(U^{(0)}(u))\right),$$ 
which is sufficient for the proposition.
\end{proof}

\begin{corollaryP}
\label{cor:fgm-main-estimate}
Assume that $\varphi$ achieves a global maximum at the repelling periodic point $\zeta$, of prime period $p$, and conditions \ref{item:U-ball} and \ref{item:repeller} hold. Let $s, t, \varsigma\in\N$ and consider $y_1,y_2,\ldots,y_\varsigma\in\R_0^+,$ $s+t<a_2<b_2<a_3<\ldots<b_\varsigma\in\N_0$. For $u$ sufficiently close to $u_F=\varphi(\zeta)$ we have
\[
\E\left(\e^{-y_1\nn_{u,0}^{s+1}- y_2\nn_{u,a_2}^{b_2}-\ldots-y_\varsigma\nn_{u,a_\varsigma}^{b_\varsigma}}\right)=\E\left(\e^{-y_1\nn_{u,0}^{s+1}}\right)\E\left(\e^{- y_2\nn_{u,a_2}^{b_2}-\ldots-y_\varsigma\nn_{u,a_\varsigma}^{b_\varsigma}}\right)+ Err,
\]
where
$
|Err|\leq C\left(s\iota(u,t)+s\sum_{j=p+1}^s\p(Q_{p,0}^0(u)\cap \{X_j>u\})+\p(U^{(0)}(u))\right),
$ 
for some $C>0$ depending only on $\theta$ given by property \ref{item:repeller} and where $\iota(u,t)$ is given by \eqref{eq:def-iota}
\end{corollaryP}
\begin{proof}
Using the same notation as in the proof of Proposition~\ref{prop:main-step} notation,  we have
\[
\E\left(\e^{-y_1\nn_{u,0}^{s+1}- y_2\nn_{u,a_2}^{b_2}-\ldots-y_\varsigma\nn_{u,a_\varsigma}^{b_\varsigma})}\right)=\sum_{\kappa_1,\kappa_2,\ldots,\kappa_\varsigma\in\N_0} \e^{-y_1\kappa_1-y_2\kappa_2-\ldots-y_\varsigma\kappa_\varsigma}\p(A^*_{\kappa_1,\underline{\kappa}})
\]
and
\[
\E\left(\e^{-y_1\nn_{u,0}^{s+1}}\right)\E\left(\e^{- y_2\nn_{u,a_2}^{b_2}-\ldots-y_\varsigma\nn_{u,a_\varsigma}^{b_\varsigma}}\right)=\sum_{\kappa_1,\kappa_2,\ldots,\kappa_\varsigma\in\N_0} \e^{-y_1\kappa_1-y_2\kappa_2-\ldots-y_\varsigma\kappa_\varsigma}\p( B^*_{\kappa_1})\p(D^{\underline{\kappa}}).
\]
Hence, 
\begin{multline*}
\left|\E\left(\e^{-y_1\nn_{u,0}^{s+1}- y_2\nn_{u,a_2}^{b_2}-\ldots-y_\varsigma\nn_{u,a_\varsigma}^{b_\varsigma})}\right)-\E\left(\e^{-y_1\nn_{u,0}^{s+1}}\right)\E\left(\e^{- y_2\nn_{u,a_2}^{b_2}-\ldots-y_\varsigma\nn_{u,a_\varsigma}^{b_\varsigma}}\right)\right|\leq \\ \sum_{\kappa_1,\kappa_2,\ldots,\kappa_\varsigma\in\N_0} \e^{-y_1\kappa_1-y_2\kappa_2-\ldots-y_\varsigma\kappa_\varsigma}|\p( A^*_{\kappa_1,\underline{\kappa}})-\p( B^*_{\kappa_1})\p(D^{\underline{\kappa}})|,
\end{multline*}
and the result follows at once from Proposition~\ref{prop:main-step}. 
\end{proof}

\begin{propositionP}
\label{prop:Laplace-estimates}
Let $X_0, X_1, \ldots$ be given by \eqref{eq:def-stat-stoch-proc-DS}, where $\varphi$ achieves a global maximum at the repelling periodic point $\zeta$, of prime period $p$, and conditions \ref{item:U-ball} and \ref{item:repeller} hold. 
 Let
$(u_n)_{n\in\N}$ be a sequence satisfying \eqref{eq:def-un}.  
Assume that conditions 
$D_p(u_n)^*$, $D'_p(u_n)^*$ hold.
Let $J\in\RR$ be \st
that $J=\bigcup_{\ell=1}^\varsigma I_\ell$ where $I_j=[a_j,b_j)\in\S$,
$j=1,\ldots,\varsigma$ and $a_1<b_1<a_2<\cdots<b_{\varsigma-1}<a_\varsigma<b_\varsigma$.
 Let
$\{u_n\}_{n\in\N}$ be such that $n\p(X_0>u_n)\to\tau>0$, as
$n\to\infty$, for some $\tau\geq 0$.  Then, for all $y_1,y_2,\ldots,y_\varsigma\in\R_0^+,$ we have
\[
\E\left(\e^{-\sum_{\ell=1}^\varsigma y_\ell\nn_{u_n}(nI_\ell)}\right)-\prod_{\ell=1}^\varsigma \E^{k_n|I_\ell|}\left(\e^{-y_\ell\nn_{u_n,0}^{n/k_n}}\right)\xrightarrow[n\to\infty]{}0
\]
\end{propositionP}
\begin{proof}
Let $h:=\inf_{j\in \{1,\ldots,\varsigma\}}\{b_j-a_j\}$ and
$H:=\lceil\sup\{x:x\in J\}\rceil=\lceil b_\varsigma\rceil$. Let $n$ be sufficiently
large so that, in particular, $k_n>2/h$ and set $\varrho_n:=\lfloor n/k_n\rfloor$.  We consider the following partition of $n[0,H]\cap
{\mathbb Z}$ into blocks of length $\varrho_n$,
$J_1=[0,\varrho_n)$, $J_2=[\varrho_n,2\varrho_n)$,\ldots,
$J_{Hk_n}=[(Hk_n-1)\varrho_n,Hk\varrho_n)$, $J_{Hk_n+1}=[Hk_n\varrho_n,Hn)$. We further cut each $J_i$ into two blocks:
$$
J_i^*:=[(i-1)\varrho_n,i\varrho_n-t_n)\; \mbox{ and } J_i':=J_i-J_i^*.
$$  
Note that
$|J_i^*|=\varrho_n-t_n$ and $|J_i'|=t_n$.

Let
$\mathscr S_\ell=\mathscr S_\ell(k)$ be the number of blocks $J_j$ contained in $\n
I_\ell$, that is,
$$\mathscr S_\ell:=\#\{j\in \{1,\ldots,Hk_n\}:J_j\subset \n I_\ell\}.$$
By assumption on the relation between $k_n$ and $h$, we have $\mathscr S_\ell>1$ for every $\ell\in \{1,\ldots,\varsigma\}$.
For each such $\ell$, we also define
$i_\ell:=\min\{j\in \{1,\ldots,k\}:J_j\subset \n I_{\ell}\}.$
Hence, it follows that  $J_{i_\ell},J_{i_\ell+1},\ldots,J_{i_\ell+\mathscr S_\ell}\subset n
I_\ell$. Moreover, by choice of the size of each block we have that
\begin{equation}
\label{eq:Sl-estimate}
\mathscr S_\ell\sim k_n|I_\ell|
\end{equation}

First of all, recall that for  every $0\leq x_i, z_i\leq 1$,  we have
\begin{equation}
\label{eq:inequality}
\left|\prod x_i-\prod z_i\right|\leq \sum |x_i-z_i|.
\end{equation}

We start by making the following approximation, in which we use \eqref{eq:inequality} and stationarity,
\begin{align*}
\left|\E\left(\e^{-\sum_{\ell=1}^\varsigma y_\ell\nn_{u_n}(nI_\ell)}\right)-\E\left(\e^{-\sum_{\ell=1}^\varsigma y_\ell\sum_{j=i_\ell}^{i_\ell+\mathscr S_\ell}\nn_{u_n}(J_j)}\right)\right|&\leq\E\left(1-\e^{-\sum_{\ell=1}^\varsigma y_\ell\nn_{u_n}(nI_\ell\setminus\cup_{j=i_\ell}^{i_\ell+\mathscr S_\ell}J_j)}\right)\\
&\leq  \E\left(1-\e^{-2\sum_{\ell=1}^\varsigma y_\ell\nn_{u_n}(J_1)}\right)\\
&\leq 2\varsigma K\E\left(1-\e^{-\nn_{u_n}(J_1)}\right),
\end{align*}
where $\max\{y_1,\ldots,y_\varsigma\}\leq K\in\N$.
In order to show that we are allowed to use the above approximation we just need to check that $\E\left(1-\e^{-\nn_{u_n}(J_1)}\right)\to0$ as $n\to\infty$.
By Corollary~\ref{cor:exponential} we have 
\begin{equation}
\label{eq:error1}
\E\left(\e^{-\nn_{u_n}(J_1)}\right)=\big(1-\varrho_n\p(Q_{p,0}^0(u_n)\big)\big)+\sum_{\kappa=1}^{\lfloor\varrho_n/p\rfloor}\e^{-\kappa}\varrho_n\big(\p(Q_{p,0}^{\kappa-1}(u_n)-\p(Q_{p,0}^\kappa(u_n)\big)+Err,
\end{equation}
where 
\[
Err\leq C\left(\varrho_n\sum_{j=p+1}^{\varrho_n}\p(Q_{p,0}^0(u_n)\cap \{X_j>u_n\})+\p(X_0>u_n)\right)\to0,
\]
as $n\to \infty$ by $D_p'(u_n)^*$ and \eqref{eq:def-un}.
Recall that  for every non negative integer $\kappa$, $\p(U^{\kappa}(u_n))\sim (1-\theta)^\kappa \p(U^{(0)}(u_n))$, which implies that $\p(Q_{p,0}^0(u_n)\sim\theta\p(X_0>u_n)$ and $\big(\p(Q_{p,0}^{\kappa-1}(u_n)-\p(Q_{p,0}^\kappa(u_n)\big)\sim \theta^2(1-\theta)^{k-1}\p(X_0>u_n)$. Applying this to \eqref{eq:error1} we get 
$
\E\left(\e^{-\nn_{u_n}(J_1)}\right)\xrightarrow[n\to\infty]{} 1, 
$
on account of \eqref{eq:def-un} again.

Now, we proceed with another approximation which consists of replacing $J_j$ by $J_j^*$. Using \eqref{eq:inequality}, stationarity and \eqref{eq:Sl-estimate}, we have
\begin{align*}
\left|\E\left(\e^{-\sum_{\ell=1}^\varsigma y_\ell\sum_{j=i_\ell}^{i_\ell+\mathscr S_\ell}\nn_{u_n}(J_j)}\right)-\E\left(\e^{-\sum_{\ell=1}^\varsigma y_\ell\sum_{j=i_\ell}^{i_\ell+\mathscr S_\ell}\nn_{u_n}(J_j^*)}\right)\right|&\leq  
\E\left(1-\e^{-\sum_{\ell=1}^\varsigma y_\ell\mathscr S_\ell \nn_{u_n}(J'_1)}\right)\\
&\leq K\sum_{\ell=1}^\varsigma \mathscr S_\ell \E\left(1-\e^{-\nn_{u_n}(J'_1)}\right)\\
&\lesssim KHk_n\E\left(1-\e^{-\nn_{u_n}(J'_1)}\right),
\end{align*}
where $\max\{y_1,\ldots,y_\varsigma\}\leq K\in\N$. Now, we must show that $k_n\E\left(1-\e^{-\nn_{u_n}(J'_1)}\right)\to0,$ as $n\to\infty$, in order for the approximation make sense. By Corollary~\ref{cor:exponential} we have 
\begin{equation}
\label{eq:error2}
\E\left(\e^{-\nn_{u_n}(J'_1)}\right)=\left(1-t_n\p(Q_{p,0}^0(u_n))\right)+\sum_{\kappa=1}^{t_n/p}\e^{-\kappa}t_n\left(\p(Q_{p,0}^{\kappa-1}(u_n)-\p(Q_{p,0}^\kappa(u_n)\right)+Err,
\end{equation}
where 
\[
k_n.\, Err\leq C\left(k_nt_n\sum_{j=p+1}^{t_n}\p\left(Q_{p,0}^0(u_n)\cap \{X_j>u_n\}\right)+k_n\p(X_0>u_n)\right)\to0,
\]
as $n\to \infty$ by $D'_p(u_n)^*$.
Hence, since by property \ref{item:repeller} of the repelling periodic point $\zeta$, we have $\p(Q_{p,0}^0(u_n))\sim\theta\p(X_0>u_n)$ and $\left(\p(Q_{p,0}^{\kappa-1}(u_n))-\p(Q_{p,0}^\kappa(u_n))\right)\sim \theta^2(1-\theta)^{k-1}\p(X_0>u_n)$, \eqref{eq:error2} gives 
\begin{equation}
\label{eq:error3}
k_n\E\left(1-\e^{-\nn_{u_n}(J'_1)}\right)\sim k_nt_n\theta\p(X_0>u_n)-k_nt_n\p(X_0>u_n)\sum_{\kappa=1}^{t_n/p}\e^{-\kappa}\theta^2(1-\theta)^{\kappa -1} \to0, 
\end{equation}
as $n\to\infty$, by \eqref{eq:def-un}.

Let us fix now some $\hat\ell\in\{1,\ldots,\varsigma\}$ and $i\in\{i_{\hat\ell},\ldots,i_{\hat\ell}+\mathscr S_{\hat\ell}\}$. Let $M_i=y_{\hat\ell}\sum_{j=i}^{i_{\hat\ell}+\mathscr S_{\hat\ell}}\nn_{u_n}(J_j^*)$ and $L_{\hat\ell}= \sum_{\ell=\hat\ell+1}^\varsigma y_\ell\sum_{j=i_{\ell}}^{i_\ell+\mathscr S_\ell}\nn_{u_n}(J_j^*).$
Using stationarity and Corollary~\ref{cor:fgm-main-estimate} along with the fact that $\iota(u_n, t)\le \gamma(n, t)$, we obtain
\begin{equation*}
\left|\E\left(\e^{-y_{\hat\ell}\nn_{u_n}(J_{i_{\hat\ell}}^*)-M_{i_{\hat\ell}+1}-L_{\hat\ell}}\right)-\E\left(\e^{-y_{\hat\ell}\nn_{u_n}(J_1^*)}\right)\E\left(\e^{-M_{i_{\hat\ell}+1}-L_{\hat\ell}}\right) \right|\leq C\,\Upsilon_n,
\end{equation*}
where 
$$\Upsilon_n=\varrho_n\gamma(n, t_n)+\varrho_n\sum_{j=p+1}^{\varrho_n}\p(Q_{p,0}^0(u_n)\cap \{X_j>u_n\})+\p(X_0>u_n).$$
Since $\E\left(\e^{-y_{\hat\ell}\nn_{u_n}(J_1^*)}\right)\leq 1$, it follows by the same argument that
\begin{align*}
\Big|\E\Big(\e^{-M_{i_{\hat\ell}}-L_{\hat\ell}}\Big)-&\E^2\Big(\e^{-y_{\hat\ell}\nn_{u_n}(J_1^*)}\Big)\E\Big(\e^{-M_{i_{\hat\ell}+2}-L_{\hat\ell}}\Big) \Big|\leq\\ 
&\quad\Big|\E\Big(\e^{-M_{i_{\hat\ell}}-L_{\hat\ell}}\Big)-\E\Big(\e^{-y_{\hat\ell}\nn_{u_n}(J_1^*)}\Big)\E\Big(\e^{-M_{i_{\hat\ell}+1}-L_{\hat\ell}}\Big) \Big|+\\
&+\E\Big(\e^{-y_{\hat\ell}\nn_{u_n}(J_1^*)}\Big)\Big|\E\Big(\e^{-M_{i_{\hat\ell}+1}-L_{\hat\ell}}\Big)-\E\Big(\e^{-y_{\hat\ell}\nn_{u_n}(J_1^*)}\Big)\E\Big(\e^{-M_{i_{\hat\ell}+2}-L_{\hat\ell}}\Big)  \Big|\\
&\qquad\qquad\qquad\qquad\qquad\qquad\quad\leq 2C\,\Upsilon_n,
\end{align*}
Hence, proceeding inductively with respect to $i\in\{i_{\hat\ell},\ldots,i_{\hat\ell}+\mathscr S_{\hat\ell}\}$, we obtain
\[
\Big|\E\Big(\e^{-M_{i_{\hat\ell}}-L_{\hat\ell}}\Big)-\E^{\mathscr S_{\hat\ell}}\Big(\e^{-y_{\hat\ell}\nn_{u_n}(J_1^*)}\Big)\E\Big(\e^{-L_{\hat\ell}}\Big) \Big|\leq
 C\mathscr S_{\hat\ell}\,\Upsilon_n.
\]
In the same way, if we proceed inductively with respect to $\hat\ell\in\{1,\ldots,\varsigma\}$, we get
\[
\left|\E\left(\e^{-\sum_{\ell=1}^\varsigma y_\ell\sum_{j=i_\ell}^{i_\ell+\mathscr S_\ell}\nn_{u_n}(J_j^*)}\right)-\prod_{\ell=1}^\varsigma \E^{\mathscr S_{\ell}}\Big(\e^{-y_{\ell}\nn_{u_n}(J_1^*)}\Big)\right|\leq C\sum_{\ell=1}^\varsigma \mathscr S_{\ell}\,\Upsilon_n.
\]
By \eqref{eq:Sl-estimate}, we have $\sum_{\ell=1}^\varsigma \mathscr S_{\ell}\,\Upsilon_n\lesssim H k_n\Upsilon_n$ and
\begin{align*}
k_n\Upsilon_n&=k_n\varrho_n\gamma(n,t_n)+k_n\varrho_n\sum_{j=p+1}^{\varrho_n}\p(Q_{p,0}^0(u_n)\cap \{X_j>u_n\})+k_n\p(X_0>u_n)\\
&\sim n\gamma(n,t_n)+n\sum_{j=p+1}^{\varrho_n}\p(Q_{p,0}^0(u_n)\cap \{X_j>u_n\})+k_n\p(X_0>u_n)\\
&\to0,
\end{align*}
as $n\to\infty$, by $D_p(u_n)^*$, $D'_p(u_n)^*$ and \eqref{eq:def-un}. 

Using \eqref{eq:inequality} and stationarity, again, we have the final approximation  
\begin{align*}
\left|\prod_{\ell=1}^\varsigma \E^{\mathscr S_{\ell}}\Big(\e^{-y_{\ell}\nn_{u_n}(J_1)}\Big)-\prod_{\ell=1}^\varsigma \E^{\mathscr S_{\ell}}\Big(\e^{-y_{\ell}\nn_{u_n}(J_1^*)}\Big)\right| &\leq K\sum_{\ell=1}^\varsigma \mathscr S_\ell \E\left(1-\e^{-\nn_{u_n}(J'_1)}\right)\\
&\lesssim KHk_n\E\left(1-\e^{-\nn_{u_n}(J'_1)}\right).
\end{align*}
Since in \eqref{eq:error3} we have already proved that $k_n\E\left(1-\e^{-\nn_{u_n}(J'_1)}\right)\to0$, as $n\to\infty$, we only need to gather all the approximations and recall \eqref{eq:Sl-estimate} to finally obtain the stated result.
\end{proof}

\begin{proof}[Proof of Theorem~\ref{thm:convergence-point-process}]
 By \cite[Theorem~4.2]{K86}, in order to prove convergence of EPP $N_n$ to the compound Poisson process $N$, it is sufficient to show that for 
any $\varsigma$ disjoint intervals $I_1, I_2,\ldots, I_\varsigma\in\S$, the joint distribution of $N_n$ over these intervals converges to the joint distribution of $N$ over the same intervals, \ie 
\[
(N_n(I_1), N_n(I_2), \ldots, N_n(I_\varsigma))\xrightarrow[n\to\infty]{}(N(I_1), N(I_2), \ldots, N(I_\varsigma)),
\]
which will be the case if the corresponding joint Laplace transforms converge. Hence, we only need to show that 
\[
\psi_{N_n}(y_1, y_2,\ldots,y_\varsigma)\to \psi_{N}(y_1, y_2,\ldots,y_\varsigma)=\e^{-\theta\sum_{\ell=1}^\varsigma (1-\phi(y_\ell))|I_\ell|}, \quad \text{as $n\to\infty$,}
\] 
for every $\varsigma$ non-negative values $y_1, y_2,\ldots,y_\varsigma$, each choice of $\varsigma$ disjoint intervals $I_1, I_2,\ldots, I_\varsigma\in\S$ and each $\varsigma\in\N$. As before, $\phi$ is the Laplace transform of the multiplicity distribution,\ie $\phi(y)=\sum_{\kappa=0}^\infty \e^{-y \kappa}\pi(\kappa)$. Note that  $\psi_{N_n}(y_1, y_2,\ldots,y_\varsigma)=\E\left(\e^{-\sum_{\ell=1}^\varsigma y_\ell N_n(I_\ell)}\right)$ and
\begin{align*}
\left|\E\left(\e^{-\sum_{\ell=1}^\varsigma y_\ell N_n(I_\ell)}\right)-\e^{-\theta\sum_{\ell=1}^\varsigma (1-\phi(y_\ell))|I_\ell|}\right|&\leq \left| \E\left(\e^{-\sum_{\ell=1}^\varsigma y_\ell N_n(I_\ell)}\right)- \prod_{\ell=1}^\varsigma \E^{k_n|I_\ell|}\left(\e^{-y_\ell\nn_{u_n,0}^{v_n/k_n}}\right)\right|\\
&+\left| \prod_{\ell=1}^\varsigma \E^{k_n|I_\ell|}\left(\e^{-y_\ell\nn_{u_n,0}^{v_n/k_n}}\right)-\e^{-\theta\sum_{\ell=1}^\varsigma (1-\phi(y_\ell))|I_\ell|} \right|
\end{align*}
Recalling that $v_n\sim n/\tau$, then, by Proposition~\ref{prop:Laplace-estimates}, the first term on the right of the previous equation goes to $0$ as $n\to\infty$. Moreover, the second term also vanishes if   
\begin{equation}
\label{eq:block-Laplace}
\E^{k_n}\left(\e^{-y\nn_{u_n,0}^{n/k_n}}\right)\xrightarrow[n\to\infty]{}\e^{-\theta\tau(1-\phi(y))},
\end{equation}
for every $y\in\R_0^+$. Hence, the result will follow as soon as we show that \eqref{eq:block-Laplace} holds. 
By Corollary~\ref{cor:k-ring}, we have
\begin{align*}
\E\left(\e^{-y\nn_{u_n,0}^{n/k_n}}\right)&=1.\p(\nn_{u_n,0}^{n/k_n}=0)+\sum_{\kappa=1}^{n/(pk_n)}\e^{-y\kappa}\p(\nn_{u_n,0}^{n/k_n}=\kappa)\\
&=1-\frac n{k_n} \p(Q^0_{p,0}(u_n))+\sum_{\kappa=1}^{n/(pk_n)}\e^{-y\kappa}\frac n{k_n}\left( \p(Q^{\kappa-1}_{p,0}(u_n))-\p(Q^{\kappa}_{p,0}(u_n)\right)+Err,
\end{align*}
where 
\[
|Err|\leq C\left(\frac n{k_n}\sum_{j=1}^{n/k_n}\p(Q_{p,0}^0(u_n),X_j>u_n)+\p(X_0>u_n)\right),
\]
for some $C>0$. Since, by condition $D'_{p}(u_n)^*$, we have that $k_n|Err|\to0$, as $n\to\infty$, and 
$\p(Q_{p,0}^\kappa)\sim\theta(1-\theta)^\kappa\p(X_0>u_n)$, then it follows that
\[
1-\E\left(\e^{-y\nn_{u_n,0}^{n/k_n}}\right)\sim\frac1{k_n}n\p(X_0>u_n)\theta\left(1-\sum_{\kappa=1}^{n/(pk_n)}\e^{-y\kappa}(1-\theta)^{\kappa-1}\right)\sim\frac1{k_n}\tau\theta\left(1-\phi(y)\right).
\]
Consequently, we have 
\[
\E^{k_n}\left(\e^{-y\nn_{u_n,0}^{n/k_n}}\right)\sim\left(1-\frac1{k_n}\tau\theta(1-\phi(y))\right)^{k_n}\to\e^{-\tau\theta(1-\phi(y))},
\]
as $n\to\infty$.
\end{proof}

\section{Applications to expanding systems}
\label{sec:HTS and Rych}

We will start this section by showing Theorem~\ref{thm:decay-L1} that gives abstract conditions in order to check $D^p(u_n)^*$ and $D'_p(u_n)^*$. We will end the section with a couple of examples that satisfy such abstract conditions. In between, we introduce notation and concepts that are needed to fit the examples into the abstract setting of  Theorem~\ref{thm:decay-L1}. 

\subsection{Mixing rates and conditions $D^p(u_n)^*$ and $D'_p(u_n)^*$}
\label{subsec:dyn-conditions}
\begin{proof}[Proof of Theorem~\ref{thm:decay-L1}]
We start by checking condition $D^p(u_n)^*$, which is straightforward. \\
Take $\phi=\I_{Q_{p,0}^{\kappa_1}(u_n)}$,
$\psi=\I_{\left(\cap_{j=2}^\varsigma \nn_{u_n}(I_j)=\kappa_j \right)}$, in \eqref{DC:L1}. By assumption, there exists $C'>0$ such that $\left\|\I_{Q_{p,0}^{\kappa_1}(u_n)}\right\|_{\mathcal C}\leq C'$, for all $n\in\N$ and $\kappa_1\in\N_0$. Hence by setting $c=CC'$, we have that condition~$D^p(u_n)^*$ holds with
$\gamma(n,t)=\gamma(t):=ct^{-2}$ and the sequence
$t_n=n^{2/3}$, for example. \\
We now turn to $D'_p(u_n)^*$. Taking $\phi=\I_{Q_p(u_n)}$ and $\psi=\I_{X_0>u_n}$ in \eqref{DC:L1} we easily get
\begin{align}
\p\left(Q_p(u_n)\cap f^{-j}(X_0>u_n)\right) &\le
 \p(Q_p(u_n))\p(X_0>u_n)+C \left\| \I_{Q_p(u_n)}\right\|_{\mathcal C}\p(X_0>u_n) j^{-2}\nonumber\\
 &\leq \p(X_0>u_n)\left(\p(Q_p(u_n))+CC'j^{-2}\right).\label{eq:computation}
\end{align}
By the Hartman-Grobman theorem there is a neighbourhood $V$ around $\zeta$ where $f$ is conjugate to its linear approximation given by the derivative at $\zeta$. Hence, for $n$ is sufficiently large so that $U_n\subset V$, if a point starts in $Q_p(u_n)$ it takes a time $\alpha_n$ to leave $V$, during which it is guaranteed that it does not return to $U_n$. Moreover, since by condition \ref{item:U-ball} and definition of $u_n$, we have that $U_n$ shrinks to $\zeta$ as $n\to\infty$, then $\alpha_n\to\infty$ as $n\to\infty$. This together with \eqref{eq:def-un}, the definition of $k_n$ and \eqref{eq:computation} implies that
\begin{align*}
n\sum_{j=1}^{[n/k_n]}\p\left(Q_p(u_n)\cap f^{-j}(X_0>u_n)\right)&=n\sum_{j=\alpha_n}^{[n/k_n]}\p\left(Q_p(u_n)\cap f^{-j}(X_0>u_n)\right)\\
&\leq \frac{n^2}{k_n}\p(X_0>u_n)\p(Q_p(u_n))+n\p(X_0>u_n)CC'\sum_{j=\alpha_n}^\infty j^{-2}\\
&\to0 \quad \text{as $n\to\infty$.}\end{align*}
\end{proof}

\begin{remark}
\label{rem:importance-L1-Dp'}
We remark that decay of correlations against $L^1$ is only crucial to prove $D'_p(u_n)^*$ since it is responsible for the factor $\p(X_0>u_n)$ in equation \eqref{eq:computation}. However, to prove $D^p(u_n)^*$ one does not need such a strong statement regarding decay of correlations. Namely, the condition would still hold if $L^1$ was replaced by $L^\infty$, for example. 
\end{remark}

\begin{remark}
\label{rem:Holder-Dp}
Note that, in the proof above, to check $D^p(u_n)^*$, it was useful that $\I_{Q_{p,0}^{\kappa_1}(u_n)}\in\mathcal C$ and $\left\|\I_{Q_{p,0}^{\kappa_1}(u_n)}\right\|_{\mathcal C}\leq C'$. However, $D^p(u_n)^*$ can still be checked even when 
$\I_{Q_{p,0}^{\kappa_1}(u_n)}\notin\mathcal C$. This is the case when $\mathcal C$ is the Banach space of H\"older observables which is used, for example, to obtain decay of correlations for systems with Young towers. Next, Proposition states that even in these cases $D^p(u_n)^*$ can still be checked.
\end{remark}

\begin{propositionP}
\label{prop:Holder-Dp}
Assume that $f:\X\to\X$ is a system with an acip $\mu$, and such that $\frac{d\mu}{\l}\in L^{1+\epsilon}$.  Assume, moreover, that the system has decay of correlations of any H\"older continuous
function $\upsilon$ of exponent $\beta$, against any $\psi\in L^{\infty}$ so that there exists some $C>0$ independent of $\upsilon,\psi$ and $n$ such that 
\begin{equation*}
\label{eq:Holder-DC}
\left| \int\upsilon\cdot(\psi\circ f^t)d\mu-\int\upsilon d\mu\int\psi
d\mu\right|\leq C\|\upsilon\|_{\mathcal H_\beta}\|\psi\|_\infty \varrho(t),
\end{equation*}
where $|\upsilon|_{\mathcal H_\beta}=\sup_{x\neq y}\frac{|\upsilon(x)-\upsilon(y)|}{|x-y|^\beta}$  and $\|\upsilon\|_{\mathcal H_\beta}=\|\upsilon\|_\infty+|\upsilon|_{\mathcal H_\beta}$.
Let $X_0, X_1, \ldots$ be given by \eqref{eq:def-stat-stoch-proc-DS}, where $\varphi$ achieves a global maximum at the repelling periodic point $\zeta$, of prime period $p$, and conditions \ref{item:U-ball}  and \ref{item:repeller} hold.   Then  $D^p(u_n)^*$ also holds.
\end{propositionP}

\begin{proof}
Essentially, we just need to follow  the proof of \cite[Lemma~3.3]{C01} and use a H\"older continuous approximation for $\I_{Q_p^{\kappa_1}(u_n)}$. The only extra difficulty is that we need an upper bound that works for all $\kappa_1\in\N_0$. However, the following trivial observation makes it possible:
\begin{multline*}
\left|\p\left(Q_{p,0}^{\kappa_1}(u_n)\cap \left(\cap_{j=2}^\varsigma \nn_{u_n}(I_j)=\kappa_j \right) \right)-\p\left(Q_{p,0}^{\kappa_1}(u_n)\right)
  \p\left(\cap_{j=2}^\varsigma \nn_{u_n}(I_j)=\kappa_j \right)\right| \\
   \le 2\p\left(Q_{p,0}^{\kappa_1}(u_n)\right)\sim\theta(1-\theta)^{\kappa_1}\p(X_0>u_n)\xrightarrow[\kappa_1\to\infty]{}0.
\end{multline*}
\end{proof}

\subsection{Equilibrium states and Banach spaces of observable functions}

Let $f:\X \to \X$ be a measurable function as above. For a measurable potential $\phi:\X\to \R$, we define the \emph{pressure} of $(\X,f,\phi)$ to be
$$P(\phi):=\sup_{\mu\in \M_f}\left\{h(\mu)+\int\phi~d\mu:-\int\phi~d\mu<\infty\right\},$$
where $h(\mu)$ denotes the metric entropy of the measure $\mu$, see \cite{W82} for details.
If, for $\mu$ is an invariant probability measure such that $h(\mu_\phi)+\int\phi~d\mu=P(\phi)$, then we say that $\mu$ is an \emph{equilibrium state} for $(\X, f, \phi)$.

 A measure $m$ is called a \emph{$\phi$-conformal} measure if $m(\X)=1$ and if whenever $f:A\to f(A)$ is a bijection, for a Borel set $A$, then
$m(f(A))=\int_A e^{-\phi}~dm$.  Therefore, setting
$$S_n\phi(x):=\phi(x)+\cdots +\phi\circ f^{n-1}(x),$$
if $f^n:A\to f^n(A)$ is a bijection then
$m(f^n(A))=\int_A e^{-S_n\phi}~dm.$

Note that for example for a smooth map interval map $f$, Lebesgue measure is $\phi$-conformal for $\phi(x):=-\log|Df(x)|$.  Moreover, if for example $f$ is a topologically transitive quadratic interval map then as in Ledrappier \cite{Le81}, any acip $\mu$ with $h(\mu)>0$ is an equilibrium state for $\phi$.  This also holds for the even simpler case of piecewise smooth uniformly expanding maps, which we consider below.  This is the case we principally consider in this paper.  For results on more general equilibrium states see \cite{FFT11}. 

We finish this section with the definition of two Banach spaces of observable functions and the respective norms  that will be used to state decay of correlations against $L^1$ for piecewise expanding maps of the interval and in higher dimensional compact manifolds.

Given a potential $\psi:Y\to \R$ on an interval $Y$, the \emph{variation} of $\psi$ is defined as
$${\rm Var}(\psi):=\sup\left\{\sum_{i=0}^{n-1} |\psi(x_{i+1})-\psi(x_i)|\right\},$$
where the supremum is taken over all finite ordered sequences $(x_i)_{i=0}^n\subset Y$.

We use the norm $\|\psi\|_{BV}= \sup|\psi|+{\rm Var}(\psi)$, which makes $BV:=\left\{\psi:Y\to \R:\|\psi\|_{BV}<\infty\right\}$ into a Banach space.

Now, let $\X$ be a compact subset of $\R^n$ and let $\psi:\X\to\R$. Given a Borel set $\Gamma\subset \X$, we define the oscillation of $\psi\in L^1(\l)$ over $\Gamma$ as
$$\mathrm{osc}(\psi,\Gamma):=\esssup\limits_{\Gamma}\psi-\essinf\limits_{\Gamma}\psi.$$

It is easy to verify that $x\mapsto \mathrm{osc}(\psi,B_{\eps}(x))$ defines a measurable function (see \cite[Proposition 3.1]{S00}). Given real numbers $0<\alpha\leq1$ and $\eps_0>0$, we define $\alpha$-seminorm of $\psi$ as
$$|\psi|_{\alpha}=\displaystyle\sup_{0<\eps\leq\eps_0}\eps^{-\alpha}\int_{\R^\N}\mathrm{osc}(\psi,B_{\eps}(x))\,\dif\l(x).$$
Let us consider the space of functions with bounded $\alpha$-seminorm
$V_\alpha=\{\psi\in L^1(\l): |\psi|_\alpha<\infty\},$
and endow $V_\alpha$ with the norm
$\|\cdot\|_\alpha=\|\cdot\|_{L^1(\l)}+|\cdot|_\alpha$
which makes it into a Banach space. We note that $V_\alpha$ is independent of the choice of $\eps_0$.

\subsection{Examples of specific dynamical systems}
\label{ssec:Examples}

\subsubsection{Rychlik systems}
\label{ssec:Rychlik}

The first  class of examples to which we apply our results is the class of  interval maps considered by Rychlik in \cite{R83}, that is given by a triple $(Y,f, \phi)$, where $Y$ is an interval, $f$ a piecewise expanding interval map (possibly with countable discontinuity points) and $\phi$ a certain potential.  This class includes, for example, piecewise $C^2$ uniformly expanding maps of the unit interval with the relevant physical measures. We refer to \cite{R83} or to  \cite[Section~4.1]{FFT12} for details on the definition of such class and instead give the following list of examples of maps in such class:
\begin{list}{$\bullet$}
{ \itemsep 1.0mm \topsep 0.0mm \leftmargin=7mm}
\item Given $m\in \{2, 3, \ldots\}$, let $f: x\mapsto mx \mod 1$ and $\phi\equiv-\log m$.  Then $m_\phi=\mu_\phi=\l$.
\item Let $f:x\mapsto 2x \mod 1$ and for $\alpha\in (0,1)$, let \begin{equation*}\phi(x):= \begin{cases}
    -\log\alpha & \text{ if } x\in (0,1/2)\\
    -\log(1-\alpha) & \text{ if } x\in (1/2,1)
    \end{cases}
    \end{equation*}
    (and $\phi=-\infty$ elsewhere).  Then $m_\phi=\mu_\phi$ is the $(\alpha, 1-\alpha)$-Bernoulli measure on $[0,1]$.
\item Let $f:(0,1] \to (0,1]$ and $\phi:(-\infty, 0)$ be defined as $f(x)=2^k(x-2^{-k})$ and $\phi(x):=-k\log2$ for $x\in (2^{-k}, 2^{-k+1}]$.  Then $m_\phi=\mu_\phi=\l$.
\end{list}
In order to prove Corollary~\ref{cor:Rychlik poiss}, we basically need to show that these systems satisfy the conditions of Theorem~\ref{thm:decay-L1}. 

In this setting, as in \cite{R83}, there is a unique $f$-invariant probability measure $\mu_\phi\ll m_\phi$ which is also an equilibrium state for $(Y,f, \phi)$ with a strictly positive density $\frac{d\mu_\phi}{dm_\phi}\in BV$. Moreover, there exists exponential decay of correlations against $L^1(m_\phi)$, \ie there exist $C>0$ and $\beta>0$, such that for any $\upsilon\in BV$ and $\psi\in L^1(m_\phi)$ we have
\begin{equation*}
\label{eq:DC-L1}
\left|\int\psi\circ f^n\cdot \upsilon~d\mu_\phi- \int\psi~d\mu_\phi \int\upsilon~d\mu_\phi\right| \le C\|\upsilon\|_{BV}\|\psi\|_{L^1(m_\phi)}\; \e^{-\beta n}.\end{equation*}
Assume  that $\zeta$ is such that $0<\frac{d\mu_\phi}{dm_\phi}(\zeta)<\infty$ and the observable $\varphi:\X\to\R\cup\{+\infty\}$ is of
the form (as considered in \cite{FFT11} and \cite[Section~3.2]{FFT12} )
 \begin{equation*}
\varphi(x)=g\left(\mu_\phi(B_{\dist(x,\zeta)}(\zeta)\right),
\end{equation*} where $\dist(\cdot)$ is a Riemannian metric on $\X$ and
the function $g:[0,+\infty)\rightarrow {\mathbb
R\cup\{+\infty\}}$ is such that $0$ is a global maximum ($g(0)$ may
be $+\infty$); $g$ is a strictly decreasing bijection $g:V \to W$
in a neighbourhood $V$ of
$0$; and has one of the three types of described in \cite[Section~1.1]{FFT10} or \cite[Section~3.1]{FFT12}. 

Assume further that $\zeta$ is a \emph{repelling $p$-periodic point}, which means that $f^p(\zeta)=\zeta$, $f^p$ is differentiable at $\zeta$ and $0<\left|\det D(f^{-p})(\zeta)\right|<1$. As shown in \cite[Theorem~5]{FFT12}, the regularity of $\mu_\phi$ and $\varphi$  guarantee that conditions \ref{item:U-ball}  and \ref{item:repeller}  of Section~\ref{subsec:point-processes} hold. Moreover, the EI is given by the formula $\theta=1-\e^{S_p\phi(\zeta)}$.

Finally, since $Q_{p,0}^{\kappa_1}(u_n)$ is the union of two intervals, for all $\kappa_1$, we have that $\|\I_{Q_{p,0}^{\kappa_1}(u_n)} \|_{BV}\leq 5$, which means all the assumptions of Theorem~\ref{thm:decay-L1} are verified and Corollary~\ref{cor:Rychlik poiss} holds.

\subsubsection{Piecewise expanding maps in higher dimensions}
\label{sssection:Saussol}
The second class of examples we consider here corresponds to a higher dimensional version of the piecewise expanding interval maps of the previous section. We refer to \cite[Section~2]{S00} for precise definition of this class of maps and give a very particular example corresponding to a uniformly expanding map on the 2-dimensional torus: 
\begin{list}{$\bullet$}
{ \itemsep 1.0mm \topsep 0.0mm \leftmargin=7mm}

\item let $\mathbb T^2=\R^2/\Z^2$ and consider the map $f:\mathbb T^2\to\mathbb T^2$ defined by the action of a $2\times2$ matrix with integer entries and eigenvalues $\lambda_1,\lambda_2>1$.

\end{list}
According to \cite[Theorem 5.1]{S00}, there exists an absolutely continuous invariant probability measure (a.c.i.p.) $\mu$. Also in \cite[Theorem 6.1]{S00}, it is shown that on the mixing components $\mu$ enjoys exponential decay of correlations against $L^1$ observables on $V_\alpha$, more precisely, if the map $f$ is as defined above and if $\mu$ is the mixing a.c.i.p., then there exist constants $C<\infty$ and $\gamma<1$ such that
\begin{equation*}\label{eq:doc-pen}
\Big|\int\psi\circ f^n\, h\, \dif\mu \Big|\leq C\|\psi\|_{L^1}\|h\|_\alpha\gamma^n,\, \forall\psi\in L^1,\mbox{ where }\int\psi\,\dif\mu=0 \textrm{ and } \forall h\in V_\alpha.
\end{equation*}
Assume that $\zeta$ is a Lebesgue density point with $0<\frac{d\mu}{d\l}(\zeta)<\infty$ and the observable $\varphi:\X\to\R\cup\{+\infty\}$ is of
the form \begin{equation}
\label{eq:observable-form} \varphi(x)=g(\dist(x,\zeta)),
\end{equation} where $\dist(\cdot)$ is a Riemannian metric on $\X$ and
the function $g:[0,+\infty)\rightarrow {\mathbb
R\cup\{+\infty\}}$ is such that $0$ is a global maximum ($g(0)$ may
be $+\infty$); $g$ is a strictly decreasing bijection $g:V \to W$
in a neighbourhood $V$ of
$0$; and has one of the three types of described in \cite[Section~1.1]{FFT10} or \cite[Section~3.1]{FFT12}. This guarantees that condition \ref{item:U-ball} of Section~\ref{subsec:point-processes} holds. Assume further that $\zeta$ is a \emph{repelling $p$-periodic point}, which means that $f^p(\zeta)=\zeta$, $f^p$ is differentiable at $\zeta$ and $0<\left|\det D(f^{-p})(\zeta)\right|<1$. Then condition \ref{item:repeller}  of Section~\ref{subsec:point-processes} holds and the EI is equal to $\theta=1-\left|\det D(f^{-p})(\zeta)\right|$ (see \cite[Theorem~3]{FFT12}) It is also easy to check that $\|\I_{Q_{p,0}^{\kappa_1}(u_n)} \|_{\alpha}$ is bounded by a positive constant, for all $\kappa_1$, which means that all conditions of Theorem~\ref{thm:decay-L1} are satisfied and thus Corollary~\ref{cor:saussol} holds.

\section{First return maps have the same statistics at periodic points as the original system}
\label{sec:BSTV}

In this section we will prove Theorem~\ref{thm:ind same}.  The conditions in that theorem concern the dynamical system $(\X,f)$ along with potentials $\phi:\X\to \R$, their conformal measures and equilibrium states.  Moreover, the observable $\varphi:\X\to [-\infty, \infty]$ must behave reasonably well.  We assume that $\X$ is a topological space with a Riemannian 
metric that we denote  by $\dist(\cdot)$. We also let $B_\eta(\zeta)$ denote an open ball of radius $\eta$ centred at $\zeta$. Moreover,
we will assume:
 
 \begin{itemize}
\item \textbf{(M1)} $P(\phi)=0$.

\item \textbf{(M2)}  There exists a finite $\phi$-conformal measure $m$ and an equilibrium state $\mu=\rho m$ with density $\rho:\X\to [1/C,C]$, some $C>0$.

\end{itemize}

Note that if $P(\phi)$ is non-zero, but finite, then $\phi$ can be replaced by $\phi-P(\phi)$ in order to satisfy (M1).  
We also require our system $(X, f, \phi)$, as well as our observable $\varphi$ to behave well around our point of interest.  We will fix a point $\zeta\in X$ where $f^p(\zeta)=\zeta$ where $p$ is the prime period.  We assume:

\begin{itemize}
\item \textbf{(S1)} 
For each $\eps_0\in (0,1)$, there exists $u_0<u_F$ such that $U(u)$ is a topological ball for any $u\in [u_0-\eps, u_F)$ and such that for each $u\in (u_0,u_F)$, there exists $n=n(u_0, u)\in \N$ and $\eps\in (0,\eps_0)$ with
$$U(u_0+\eps)\subset f^{np}(U(u)) \subset U(u_0-\eps).$$

\item \textbf{(S2)} For some $\delta>0$, $f^p|_{B_\delta(\zeta)}$ is a bijection onto its image and  \begin{equation*}
\sup_{|x-\zeta|<\delta}\sum_{n=0}^\infty |\phi(x_n)-\phi(\zeta)|<\infty
\label{eq:smooth}
\end{equation*}
where $x_n\in B_\delta(\zeta)$  are such that $f^{pn}(x_n)=x$.  Moreover, $\phi(\zeta)<0$.

\end{itemize}

\begin{remark}
A natural example where these conditions hold is the following.  Suppose that $\X$ is an interval $I$ and $f^p:\X\to \X$ is a $C^1$ expanding map at $\zeta$ and $\varphi$ is a potential as in \eqref{eq:observable-form} (similarly if $\X$ is a subset of the complex plane and  $f^p$ is holomorphic at $\zeta$ with $|Df^p(\zeta)|>1$).  Then (S1) holds.
In these two cases condition (S2) holds if  $\phi:I\to \R$ is locally H\"older at $\zeta$, a particular example is if $\phi(x)=-\log|Df(x)|$.
\end{remark}

Before embarking on the proof of  Theorem~\ref{thm:ind same}, we prove a lemma which demonstrates how, given (S1) and (S2), the conformal measure scales around $\zeta$.  

\begin{lemma}
Suppose that $(\X, f, \phi)$ is a dynamical system satisfying (M1), (M2), (S1) and (S2).  Then for any  $u_0<u_F$ as in (S1), for $u\in (u_0,u_F)$ and $n$ as in (S1),  $$m_\phi(U(u))\asymp e^{nS_p\phi(\zeta)} m_\phi(U(u_0)).$$
\label{lem:ball meas}
\end{lemma}

Note that in the case that $f$ is a $C^1$ interval map, $\varphi=-\dist(\cdot, \zeta)$ and $\phi=-\log|Df|$, with $m_\phi$ being Lebesgue measure, this lemma implies the elementary fact that for small $\delta>0$ the Lebesgue measure of the ball $B_{\lambda\delta}(\zeta)$ for $\lambda=|Df^p(\zeta)|^{-n}$ is approximately the same as $|Df^p(\zeta)|^{-n}$ times the Lebesgue measure of $B_\delta(\zeta)$. 

\begin{proof} We fix $u_0<u_F$ and $\eps_0>0$ compatible with (S1).  (S2) implies that for $u$ close to $u_F$,
$$m_\phi\left(f^{np}(U(u))\right)=\int_{U(u)}e^{-S_{np}\phi}~dm_\phi\asymp \int_{U(u)} e^{-n S_p\phi(\zeta)}~dm_\phi= e^{-nS_p\phi(\zeta)} m_\phi(U(u)).$$
Since  (S1) implies that there exists $\eps\in (0,\eps_0)$ such that for $u\in(u_0, u_F)$ there exists $n\in \N$ such that $U(u_0+\eps)\subset f^{np}(U(u)) \subset U(u_0-\eps)$ , treating all such sets as having approximately the same measure, we obtain
$m_\phi(U(u))\asymp e^{nS_p\phi(\zeta)} m_\phi(U(u_0))$, as required.
\end{proof}

\begin{proof}[Proof of  Theorem~\ref{thm:ind same}]
For this proof, we will drop the $\phi$ subscript for $\mu, \ m$ and $\rho$.
Recall that $\hat\mu=\mu_{\hat \X}$ and for any $V\subset \hat \X$, the function  $\hat r_V$ is the first hitting time to $V$ by $\hat f$. 

We will fix some scale $u_0<u_F$ satisfying (S1) and define $\nn_u$ and $v$ as in \eqref{eq:def-REPP}. 

We will prove the theorem for returns rather than hits, that is, we consider
\begin{equation*}
\lim_{u\to u_F} \mu_{U(u)}\left(\left\{\nn_u(v J)= \kappa\right\}\right)
\label{eq:Pois_first}
\end{equation*}
rather than $\lim_{u\to u_F} \mu\left(\left\{\nn_u(v J)= \kappa\right\}\right)$. The reason we do so is because, in this setting, we start on $U(u)$ for both the induced and the original map and this helps to obtain a relation between  induced return times and return times to $U(u)$. If we were to consider hitting times rather than return times then we would have to start on $\hat \X$ or on $\X$ depending on whether we were considering the induced map or the original map, respectively. 

Our hypothesis is that, for every $J\in\mathcal S$ and $\kappa\in\N_0$,
\[
\lim_{u\to u_F} \hat\mu_{U(u)}\left(\left\{\hat\nn_u(\hat v J)\leq \kappa\right\}\right)
=\Law(J,\kappa),
\]
where $\Law$ is such that $\Law(J,\cdot)$ corresponds to a d.f. of an integer valued r.v., $\lim_{\delta\to0}\Law((1\pm\delta)J,\kappa)=\Law(J,\kappa)$, for every $\kappa$. 
It is easy to check that, for the limiting compound Poisson process with a cluster at time $0$, we have
\begin{equation}
\label{eq:relation-process-return}
\lim_{t\to 0}\Law([0,t),1)=\lim_{u\to u_F}\frac{\mu(Q^0(u))}{\mu(U(u))}.
\end{equation}
We want to show that, for every $J\in\mathcal S$ and $\kappa\in\N_0$,
\[
\lim_{u\to u_F} \mu_{U(u)}\left(\left\{\nn_u(v J)\leq \kappa\right\}\right)
=\Law(J,\kappa).
\]
This relation between the induced and the original return times point processes can be converted back to hitting times by applying the results in \cite{CK06,HLV07}.\footnote{The relation in \eqref{eq:HTS-RTS} between the distribution of the first hitting time and the first return time was extended in  \cite{CK06} to the $k$-th hitting and return time, which were related by a similar integral equation. Then, based on this result, in \cite{HLV07} a similar relation between the limits for the Hitting and Return Times point processes was proved.} 

Let
$$E_n(x):=\frac1n\sum_{i=0}^{n-1}r_{\hat \X}\circ\hat f^i(x).$$  For $\mu$-a.e. $x\in \hat \X$,
$$E_n(x)\to c:=\int_{\hat \X}r_{\hat \X}~d\mu=\frac1{\mu(\hat \X)}$$
where the final equality follows from Kac's Theorem.

For $\mu$-a.e. $x\in \hat \X$, there exists a finite number $j(x, \eps)$ such that  $|E_n(x)-c|<\eps$ for all $n\ge j(x, \eps)$.  Let $\tilde G_{n}^\eps:=\{x\in \hat \X:j(x,\eps)<n\}$.  Moreover, we define $N=N(\eps)$ to be such that $\hat\mu(\tilde G_N^\eps)>1-\eps$.

Since $$\left|\sum_{i=0}^{n-1}r_{\hat \X}(\hat f^i(x))-cn\right|<\eps n \text{ for } x\in \tilde G_N^\eps \text{ and } n\ge N,$$
for all such $n$, there exists $s=s(x)$ with $|s|<\eps n$ such that $\hat f^n(x)=f^{cn+s}(x)$.  Then we have
$$r_{U(u)}(x)=c\hat r_{U(u)}(x)+s$$
for some $|s|<\eps\hat r_{U(u)}(x)$ whenever $\hat r_{U(u)}(x)\ge N$ and $x\in \tilde G_N^\eps$.

We again use the notation for the hierarchy of balls and annuli around $\zeta$ denoted $U^\kappa(u)$ and $Q^\kappa(u)$, respectively, where $U^0(u)$ corresponds to an exceedance of the threshold $u$.    

For $\eps, {\eta}_0>0$,
$$ G_{N,n}^\eps=G_{N,n}^{\eps, u_0}:=\left\{x\in Q^n(u_0): f^{np}(x)\in G_N^\eps\right\} .$$
We use the idea that we can fix some scale $u_0$ and have $G_N^\eps$ `sufficiently dense' in $Q^0(u_0)$, and then pull back by $n(u_0, u)$ (as in (S1)) steps to get points in $G_{N, n}^\eps$ being sufficiently dense in $Q^{0}(u)$.  We let  $G_{N, n}=G_{N, n}^\eps$. 

\begin{lemma}
There exist $\gamma, u_1>0$ depending on (M2), (S1) and (S2) such that:
\begin{enumerate}
\item   For each $\epsilon>0$ and $u\in [u_1,u_F)$, there exists $N\in \N$ such that $\mu_{Q^0(u)}(G_N)>1-\epsilon$.
\item Fix $u_0<u_F$ and $\eps_0>0$ satisfying (S1) and so that $u_0-\eps_0\geq u_1$.  If $\epsilon>0$ and $u_0$ and $N$ are as in (1), then for each $u\in (u_0,u_F)$, for $n$ given by (S1), 
  $$\mu_{Q^0(u)}\left(G_{N,n}\right)>1-\gamma\epsilon.$$
\end{enumerate}
\label{lem:large scale alt}
\end{lemma}

\begin{proof}
We fix some $u_0<u_F$ and $\eps_0>0$  as in (S1).  Let $u$ be any value in $[u_0-\eps_0, u_F)$.  The fact that we can choose $N$ large enough that $\mu_{Q^0(u)}(G_N^c)\le \epsilon$ follows from the ergodic theorem, so (1) is immediate.  

The bound $C$ on the density $\rho$ given in (M2) implies that $\frac{m_\phi(G_N^c\cap Q^0(u_0))}{m_\phi(Q^0(u_0))}\le C^2\epsilon$.  By (S1), $U(u_0+\eps_0)\subset f^{np}(U(u)) \subset U(u_0-\eps_0)$.    By (1) and (S2) there exists $C_\phi>0$ such that
\begin{align*}
m_\phi(G_N^c\cap Q^0(u_0)) &= C_\phi^\pm \int_{f^{-np}(G_N^c)\cap Q^0(u)}e^{-S_{np}\phi}~dm_\phi\\
& =C_\phi^{\pm 2} m_\phi\left(f^{-np}(G_N^c)\cap Q^0(u)\right)e^{-S_{np}\phi(\zeta)}\\
&=C_\phi^{\pm 2}  m_\phi\left(G_{N,n}^c\cap Q^0(u)\right)e^{-S_{np}\phi(\zeta)}.
\end{align*}

Similarly \begin{equation*}
m_\phi(Q^0(u))= C_\phi^{\pm 2}  m_\phi(Q^0(u_{0})) e^{S_{np}\phi(\zeta)}.\label{eq:Br_n}
\end{equation*}
Hence $\mu_{Q^0(u)}(G_{N,n}^c)<C_\phi^4C^4\epsilon$.
Putting these facts together, along with the density bound, implies that the lemma holds  with $\gamma=C_\phi^4C^4$. 
\end{proof}

In the proof of this lemma, we implicitly used the set $\tilde Q^0(u_0)$ which we now define as the $f^{np}(Q^0(u))$.  Moreover, we set $\tilde Q^i(u_0):=f^{np}(Q^i(u))$ (note that for $i>n$, $\tilde Q^i(u_0)=Q^{i-n}(u)$), and $\tilde G_{N, n}=\left\{x\in \tilde Q^n(u_0): f^{np}(x)\in G_N^\eps\right\}$.  Note that the lemma also holds for these sets.

For each $x\in\{\nn_u(v J)(x)=\kappa\}$, let $\alpha(\kappa)(x)$ denote the number of escaping returns among the total amount of returns, $\kappa\in\N_0$, i.e., returns to $Q^0(u)$. This means that there exist $\alpha(\kappa)$ clusters, with sizes: $\beta_1,\ldots,\beta_{\alpha(\kappa)}$, such that $\beta_1+\ldots+\beta_{\alpha(\kappa)}=\kappa$. Also, for any $x\in U(u)$ and $i=2,\ldots,\kappa$, let $r_{U(u)}^{(i)}(x):= r_{U(u)}(f^{r_{U(u)}^{(i-1)}})(x)$, where $r_{U(u)}^{(1)}(x)=0$, since we start on $U(u)$. Observe that the escaping returns correspond to the indices $i$ for which there exist $j$ such that $i=\beta_1+\ldots+\beta_j$. Also note that we can split the returns in two classes: 
\begin{itemize}
\item the first one corresponds to the \emph{intra-cluster} returns: one of a series of returns ending with the respective escaping return, i.e.,  corresponding to the indices $i$ for which there exists $1\leq j\leq \alpha(\kappa)-1$ such that $\beta_1+\ldots+\beta_j+1<i\leq\beta_1+\ldots+\beta_{j+1}$;

\item  the second one refers to the \emph{inter-cluster} returns: one of the returns in the time gap between two clusters, \ie corresponding to the indices  $i$ for which there exists $1\leq j\leq \alpha(\kappa)-1$ such that $i=\beta_1+\ldots+\beta_j+1$. We shall denote these inter-cluster returns by $i_j:=\beta_1+\ldots+\beta_j+1$, with $1\leq j\leq \alpha(\kappa)-1$.

\end{itemize}

For the induced map, we can define $\hat r_{U(u)}^{(i)}$ in the natural way. Our goal is to relate the the return times $r_{U(u)}^{(i)}$ of the original system with the corresponding ones of the induced map. We recall that our choice of $\hat X$ was such that $\zeta$ is the only point of its orbit belonging to $\hat X$. Hence, for the intra-cluster returns the relation is obvious:
\begin{equation}
\label{eq:intra-cluster}
r_{U(u)}^{(i)}=\hat r_{U(u)}^{(i)}+(p-1).
\end{equation}
The hard part is to get a relation for the intra-cluster returns. For these we will use the ergodic theorem to show that $r_{U(u)}^{(i_j)}$ is approximately $c\hat r_{U(u)}^{(i_j)}$, where $c$ is the expected first return time to $\hat \X$.

Observe that in order to have a cluster of size $\beta$, there must be an entrance in $\tilde Q^{n+\beta}(u_0)$. For every $i=1,\ldots, \alpha(\kappa)$, we define $\beta_i^*=n+\beta_i$ and write:
\begin{equation*}
\{x\in U(u): \nn_u(v J)(x)=\kappa\}=\bigcup_{\Small\begin{tabular}{c}
$0\leq\alpha(\kappa)\leq \kappa$\\
$\sum_{j=1}^{\alpha(\kappa)}\beta_j=\kappa$
\end{tabular}} U(u)\cap \{\nn_u(v J)=\kappa\}\cap \left(\bigcap_{j=1}^{\alpha(\kappa)}f^{-r_{U(u)}^{(i_j)}}\left(\tilde Q^{\beta_{j+1}^*}(u_0)\right)\right)
\end{equation*}

Now, for each $j=1,\ldots,\alpha(k)-1$, we may split $\tilde Q^{\beta_j^*}(u_0)$ in the following way:
\begin{equation*}
\tilde Q^{\beta_j^*}(u_0)= \left(\tilde Q^{\beta_j^*}(u_0) \cap \tilde G_{N, \beta_j^*}\cap \{\hat r_{U(u)}\circ f^{\beta_j^*p}\ge N\}\right)\cup\left(\tilde Q^{\beta_j^*}(u_0)\cap\left( \tilde G_{N,\beta_j^*}^c\cup \{\hat r_{U(u)}\circ f^{\beta_j^*p}< N\}\right)\right)
\end{equation*}

Using Lemma~\ref{lem:large scale alt}, for any $\epsilon>0$, we can choose $N$ sufficiently large so that 
\begin{equation}\label{eq:GNc}
\mu_{U(u)}\left(\tilde Q^{\beta_j^*}(u_0)\cap\tilde G_{N,\beta_j^*}^c\right)\leq\mu_{\tilde Q^{\beta_j^*}(u_0)}\left(\tilde G_{N,\beta_j^*}^c\right)<\epsilon,
\end{equation}
independently of $n$ and $\beta_i$.

Moreover, for  $x\in \tilde Q^{\beta_j^*}(u_0)$, since $f^{\beta_jp}(x)\in Q^0(u)$, we have $r_{U(u)}\circ f^{\beta_jp}(x)=r_{U(u)}\circ f^{\beta_j^*p}(x)+pn$ (alternatively $\hat r_{U(u)}\circ f^{\beta_jp}(x)=\hat r_{U(u)}\circ f^{\beta_j^*p}(x)+n$).  Therefore using the idea of Lemma~\ref{lem:large scale alt} again,
\begin{align*}
\mu_{\tilde Q^{\beta_j^*}(u_0)}\left( \{\hat r_{U(u)}\circ f^{\beta_j^*p}\geq N\}\right) &\sim \mu_{
Q^0(u)}\left( \left\{\hat r_{U(u)}> N+n\right\}\right)
\end{align*}
Using the facts that $\mu(U(u))(N+n)\asymp ne^{nS_p(\zeta)}$ where $n=n(u_0, u)$ as $n\to \infty$ as $u\to u_F$, $S_p(\zeta)<0$ and \eqref{eq:relation-process-return} it follows that
\begin{align*}
\lim_{u\to u_F}\mu_{{Q^0(u)}}\left( \left\{\hat r_{U(u)}>N+n\right\}\right)&=\lim_{u\to u_F}\mu_{U(u)}\left( \left\{\hat r_{U(u)}> \frac{\mu(U(u))(N+n)}{\mu(U(u))}\right\}\right) \lim_{u\to u_F}\frac{\mu(Q^0(u))}{\mu(U(u))}\\
&=\lim_{t\to 0} \Law([0,t),1) \lim_{u\to u_F}\frac{\mu(Q^0(u))}{\mu(U(u))}=1.\end{align*}
Therefore for every $\epsilon>0$ for $u$ close enough to $u_F$, we have that 
\begin{equation}
\label{eq:N}
\mu_{U(u)}\left(\tilde Q^{\beta_j^*}(u_0)\cap \{\hat r_{U(u)}\circ f^{\beta_j^*p}< N\}\right)\leq\mu_{\tilde Q^{\beta_j^*}(u_0)}\left(\{\hat r_{U(u)}\circ f^{\beta_j^*p}< N\}\right) <\epsilon.
\end{equation}
Consider the event
$$
E(\kappa,u,J):= \hspace{-3mm} \bigcup_{\tiny\begin{tabular}{c}
$0\leq\alpha(\kappa)\leq \kappa$\\
$\sum_{j=1}^{\alpha(\kappa)}\beta_j=\kappa$
\end{tabular}}
\hspace{-4mm} \{\nn_u(v J)=\kappa\}\cap \left(\bigcap_{j=1}^{\alpha(\kappa)}f^{-r_{U(u)}^{(i_j)}}\left(\tilde Q^{\beta_{j}^*}(u_0)\cap \tilde G_{N, \beta_j^*}\cap \{\hat r_{U(u)}\circ f^{\beta_j^*p}\ge N\}\right)\right)
$$
Since $\alpha(\kappa)\leq \kappa$, by stationarity, \eqref{eq:GNc} and \eqref{eq:N}, for $N$ sufficiently large  and $r$ sufficiently small we have 
\begin{multline}
\Big|\mu_{U(u)}(\nn_u(v J)=\kappa)-\mu_{U(u)} \left(E(\kappa,u,J)\right)\Big|\\ \leq \kappa\mu_{U(u)}\left(\tilde Q^{\beta_j^*}(u_0)\cap \tilde G_{N,\beta_j^*}^c\right)+\kappa\mu_{U(u)}\left(\tilde Q^{\beta_j^*}(u_0)\cap \{\hat r_{U(u)}\circ f^{\beta_j^*p}< N\}\right)\leq 2\kappa\epsilon
\label{eq:approx1linha}
\end{multline}

Moreover, for $x\in E(\kappa,u,J)$, by definition of $\tilde G_{N,0}$, it follows that for every $j=1,\ldots, \alpha(\kappa)-1$, there exists $|s_j|<\eps \hat r_{U(u)}^{(i_j)}(x)$ such that
\begin{equation}\label{eq:inter-cluster}r_{U(u)}^{(i_j)}(x)=c\hat r_{U(u)}^{(i_j)}(x)+s_j.\end{equation}  
Since $\hat v=v/c$, from \eqref{eq:inter-cluster}, we easily get that for $x\in E(\kappa,u,J)$ and every $j=1,\ldots, \alpha(\kappa)-1$
\begin{equation}
\label{eq:rel1}
r_{U(u)}^{(i_j)}(x)\in v_\eta J\quad\Rightarrow\quad \hat r_{U(u)}^{(i_j)}(x)\in\hat v(1+\eps/c)J
\end{equation}
and
\begin{equation}
\label{eq:rel2}
\hat r_{U(u)}^{(i_j)}(x)\in\hat v(1-\eps/c)J\quad\Rightarrow\quad r_{U(u)}^{(i_j)}(x)\in v J.
\end{equation}
Besides, since for $u$ sufficiently close to $u_F$, we have $p-1\leq\hat v\eps/c J_{\sup}$, then, by \eqref{eq:intra-cluster}, relations \eqref{eq:rel1} and \eqref{eq:rel2} hold for all $i=1,\ldots,\kappa$. Hence,
$$
\mu_{U(u)}(\hat\nn_u(\hat v(1-\eps/c)J)=\kappa)\leq\mu_{U(u)} \left(E(\kappa,u,J)\right)\leq\mu_{U(u)}(\hat\nn_u (\hat v (1+\eps/c)J)=\kappa).
$$
Recalling that $\mu_{U(u)}=\hat\mu_{U(u)}$, taking limits as $u\to u_F$, by hypothesis, we get that $$\Law((1-\eps/c)J,\kappa) -\Law((1-\eps/c)J,\kappa-1)\leq\mu_{U(u)} \left(E(\kappa,u,J)\right)\leq \Law((1+\eps/c)J,\kappa) -\Law((1+\eps/c)J,\kappa-1).$$ 

Finally, using \eqref{eq:approx1linha} and that $\lim_{\delta\to0}\Law((1\pm\delta)J,\kappa)=\Law(J,\kappa)$, we get for all $\kappa\in\N_0$
$$
\lim_{u\to u_F}\mu_{U(u)}(\nn_u (v J)=\kappa)=\Law(J,\kappa).
$$

Moreover, the limit above can be shown to be uniform in $\kappa$. To that end, notice that
\begin{align*}
\mu_{U(u)}(\nn_u (v J)>\kappa)&\leq \mu_{U(u)}\left(\nn_u (v [0,J_{\sup}))>\kappa\right)\\
&\leq\mu_{U(u)}\left(\hat \nn_u (v [0,J_{\sup}))>\kappa\right)\\
&=\mu_{U(u)}\left(\hat \nn_u(\hat v c[0,J_{\sup}))>\kappa\right)\xrightarrow[u\to u_F]{}1- \Law\left([0,cJ_{\sup}),\kappa\right)\xrightarrow[\kappa\to\infty]{}0.
\end{align*}
This implies that we can choose $K(\eps)$ such that $\mu_{U(u)}\left(\nn_u (v J)>K(\eps)\right)<\eps$. Then we only have to consider that $N$ is sufficiently large and $u$ is sufficiently close to $u_F$ so that $\epsilon$ from \eqref{eq:GNc} and \eqref{eq:N} is such that $\epsilon<\eps/K(\eps)$ to conclude that approximation \eqref{eq:approx1linha} does not depend on $\kappa$ anymore. 
\end{proof}

\section{Extending Poisson statistics at periodic points beyond systems with good first return maps}

Given  \cite[Proposition 1]{FFT11} and Theorem~\ref{thm:ind same}, the proof of Theorem~\ref{thm:BrV new} is almost identical to that of \cite[Theorem 2]{BV03}, the main difference being that in that paper they were interested in $z$ being a typical point of $\mu$ (i.e., they needed the result to hold for $\mu$-a.e. $z\in I$), while here we are picking a specific point.  The only place where this issue arises in \cite{BV03} is in Lemma 4 of that paper, where the summability condition guarantees that \eqref{eq:dens} automatically holds at $\mu$-a.e. point.  This is why we need to add  \eqref{eq:dens}  to our assumptions.
Due to these strong similarities with a previous work, we only sketch the proof here.   Moreover, since the argument is the same for the Poisson statistics, we only focus on the first hitting time distribution.

\begin{remark}
Note that \cite[Lemma 10]{BT09} extends \cite[Lemma 4]{BV03} to remove the necessity of the summability condition.  However, the proof of that result used the fact that they were only interested in typical points of $\mu$ rather strongly, so that method seems unlikely to extend to our setting here.
\label{rmk:BruTod}
\end{remark}

The map $\hat f$ is not a first return map in this case; indeed no first return map will be a Rychlik map as in  \cite[Proposition 1]{FFT11}.  Instead the system is lifted to a Hofbauer extension/tower (see\cite{H80, K89} ): there is a countable collection of intervals $\D=\{D_k\}$ and a set $\tilde I=\sqcup_{D_k\in \D} D_k$ with a map $\tilde f:\tilde I \to \tilde I$ semiconjugate to $f$ by a projection map $\pi$, i.e., the following diagram commutes:
\[
\begin{array}{ccc}
\tilde I & \stackrel{\tilde f}{\rightarrow} & \tilde I \\
\pi \downarrow  & & \downarrow \pi \\
I &  \stackrel{f}{\longrightarrow} &
I
\end{array}
\]
By \cite{K89}, for any ergodic invariant probability measure $\nu$ on $I$ with positive entropy, there is an ergodic invariant probability measure $\tilde \nu$ on $\tilde I$ such that $\nu=\tilde\nu\circ\pi^{-1}$.

The system $(\tilde I, \tilde f)$ has a Markov structure which means that for any interval $\tilde{\mathcal U}$ compactly contained in a set $D\in \D$, the first return map to $F_{\tilde{\mathcal U}}:\cup_i\tilde{\mathcal U}_i\to \tilde{\mathcal U}$ is such that for each $i\in \N$, $\tilde F_{\tilde{\mathcal U}}:\tilde{\mathcal U}_i \to \tilde{\mathcal U}$ is a diffeomorphism (note that we can also choose $\tilde{\mathcal U}$ so that the sets $\{\tilde{\mathcal U}_i\}_i$ do not overlap).  Moreover, this map has bounded distortion: there exists $K\ge 1$ such that for any $i\in \N$, $\frac{DF_{\tilde{\mathcal U}} (x)}{DF_{\tilde{\mathcal U}}(y)} \le K$ for any $x,y\in \tilde{\mathcal U}_i$.  In \cite{BV03}, the set $\tilde{\mathcal U}$ is chosen to be a certain union of such intervals.   Defining $\mathcal U:=\pi(\tilde{\mathcal U})$, we let $F_\mathcal U(x)=\pi \circ F_{\tilde{\mathcal U}}(\tilde x)$ where $\pi(\tilde x)=x$ (the special way that $\tilde{\mathcal U}$ was chosen means that any such $\tilde x$ gives the same value).  Then as in \cite[Lemma 3]{BV03}, $F_\mathcal U$ is a Rychlik map.  Hence by \cite[Proposition 1]{FFT11} $F_\mathcal U$ has RTS $e^{-\theta t}$ at $\zeta$.  Since $F_\mathcal U$ is not, in general, a \emph{first} return map for all points in $\mathcal U$, we have to do a bit more work before concluding that the original map has the same statistics.

Let $Q^0 (u,\tilde{\mathcal U}):=\pi^{-1}(Q^0(u))\cap\tilde{\mathcal U}$.  Above we have shown that the distribution of  the normalised first return time $r_{Q^0 (u,\tilde{\mathcal U})}$ converges to $e^{-\theta t}$ as $\eta\to 0$.  Now for $x\in Q^0(u)$, let $\tilde r_{Q^0 (u,\tilde{\mathcal U})}(x)=r_{Q^0(u,\tilde{\mathcal U})}(\tilde x)$ where $\tilde x\in Q^0 (u,\tilde{\mathcal U})$ has $\pi(\tilde x)=x$ (again any such point suffices).
The analogy of \cite[Lemma 4]{BV03} in our case is the statement that
$$\mu_{Q^0(u)}\left(\left\{x:\tilde r_{Q^0(u,\tilde{\mathcal U})}(x) \neq r_{Q^0(u)}(x)\right\}\right)\to 0$$ as we shrink both $\eta$ to zero and our set $\tilde{\mathcal U}$ so that $\mathcal U=\pi(\tilde{\mathcal U})$ shrinks to $z$.  Thus the normalised distribution of $r_{Q^0(u)}$ also converges to $e^{-\theta t}$ as $\eta r\to 0$, as required.

\section{Exclusion of parameters for quadratic interval maps}
\label{sec:BC arg}

In this section we show how we can adapt the procedure of exclusion of parameters of Benedicks-Carleson for the quadratic map, in order to guarantee that condition \eqref{eq:dens} holds for periodic points on a positive Lebesgue measure set of parameters. We consider the quadratic family of maps $f_a:[-1,1]\to[-1,1]$, given by $f_a(x)=1-ax^2$ where $a<2$ is a parameter close to the value $a=2$, for which the orbit of the critical point $0$ ends up at the fixed point $-1$. For that reason we call $f_2$ the full quadratic map. In what follows will use $D$ to denote the derivative with respect to $x$, so, for example, $Df_a(x)=\frac{d (f_a(x))}{dx}=-2a x$.

Let $\zeta_2$ be a hyperbolic repelling point of $f_2$, i.e., there exists $q\in \N$ such that $f_2^q(\zeta_2)=\zeta_2$ and $|Df_2^q(\zeta_2)|>1$.  Then there exists $a_0<2$ such that for  $a\in (a_0, 2]$ there is a corresponding \emph{hyperbolic continuation} $\zeta_a$ such that $f_a^q(\zeta_a)=\zeta_a$,  $|Df_a^q(\zeta_a)|>1$ and $a\mapsto \zeta_a$ is analytic in $(a_0, 2)$.  The following theorem says that given such a point $\zeta_2\in (-1,1)$, not in the critical orbit we can find a large set of maps $f_a$ which have a hyperbolic continuation $\zeta_a$ of $\zeta_2$ and such that the density of the acip $\mu=\mu_a$ is bounded at $\zeta_a$.  (Note that the density of $\mu_2$ is clearly bounded at $\zeta_2$.)

\begin{theorem}
Let $\zeta_2$ be a repelling periodic point of $f_2$, of prime period $q\in\N$ and distinct from the critical orbit. Consider its hyperbolic continuation $\zeta_a$ for $a$ sufficiently close to the parameter value 2. Then there exists a positive Lebesgue measure set of parameters $\Omega_\infty=\Omega_\infty(\zeta_2)$, for which $f_a$ has an acip and such that condition \eqref{eq:dens} holds for $\zeta_a$, for all $a\in\Omega_\infty$. 
\label{thm:BC}
\end{theorem}

\begin{corollary}
Let $\zeta_2$ be a repelling periodic point of $f_2$, of prime period $q\in\N$ and distinct from the critical orbit. Consider its hyperbolic continuation $\zeta_a$ for $a$ sufficiently close to the parameter value 2. Then there exists a positive Lebesgue measure set of parameters $\Omega_\infty=\Omega_\infty(\zeta_2)$ such that the REPP converges to a compound Poisson process, as in Theorem~\ref{thm:convergence-point-process}, with $\theta=1-\frac1{|Df^q(\zeta_a)|}$, for all $a\in\Omega_\infty$. 
\end{corollary}

\begin{proof}
This is immediate from Theorems~\ref{thm:ind same} and \ref{thm:BrV new} since \eqref{eq:dens} holds at all such $\zeta_a$.
\end{proof}

The Benedicks-Carleson Theorem (see \cite{BC85} or Section 2 of
\cite{BC91}) states that there exists a positive Lebesgue measure
set of parameters, $\BC$, verifying 
\begin{align}
&\mbox{there is $c>0$ ($c\approx \log2$)
such that $|Df_a^n(f_a(0))|\geq \e^{cn}$ for all $n\in\N_0$};
\tag{EG}\label{eq:EG}\\
&\text{there is a small $\alpha>0$ such that $|f_a^n(0)|\geq\e^
{-\alpha n}$ for all $n\in\N$}\tag{BA}\label{eq:BA}.
\end{align}
The condition
\eqref{eq:EG} is usually known as the Collet-Eckmann condition which
was introduced in \cite{CE83} and  admits, among other things, the proof of the existence of an acip, which was the major goal of the celebrated paper of Jakobson \cite{J81}.  We will adapt the Benedicks-Carleson argument, using the presentation of their construction in \cite{M93}, which particularly suits our purposes here. 

We define the \emph{critical region} as the
interval $(-\delta,\delta)$, where $\delta=\e^{-\Delta}>0$ is chosen
small, but much larger than $2-a$. This region is partitioned into
the intervals
$
(-\delta,\delta)=\bigcup_{m\geq\Delta} I_m,
$
where $I_m=(\e^{-(m+1)},\e^{-m}]$ for $m>0$ and
$I_m=[-e^m, -e^{m-1})$ for
$m<0$; then each $I_m$ is further subdivided into $m^2$ intervals
$\{I_{m,j}\}$ of equal length inducing the partition $\P$ of
$[-1,1]$ into
$ [-1,-\delta)\cup
\bigcup_{m,j}I_{m,j}\cup(\delta,1].
$ For definiteness, the smaller the $j=1,\ldots, m^2$ the closer $I_{m,j}$ is to the critical point.
Given $J\in\P$, let $nJ$ denote the interval $n$ times the length of $J$ centred at $J$ and define $I_{m,j}^+=3I_{m,j}$ and $\mathscr U_m:=(-\e^{-m},\e^{-m})$, for every
$m\in\N$. 

\subsection{Expansion outside the critical region}
\label{subsec:free-period-estimates} There is $c_0>0$ and $M_0\in\N$
such that for all $a$ sufficiently close to $2$ we have
\begin{enumerate}

\item \label{item:free-period-M0-1d} If
$x,\ldots,f_a^{k-1}(x)\notin(-\delta,\delta)$ and $k\geq M_0$, then
$|Df_a^k(x)|\geq\e^{c_0k}$;

\item \label{item:free-period-return-1d} If
$x,\ldots,f_a^{k-1}(x)\notin(-\delta,\delta)$ and
$f_a^k(x)\in(-\delta,\delta)$, then $|Df_a^k(x)|\geq\e^{c_0k}$;

\item \label{item:free-period-1d} If
$x,\ldots,f_a^{k-1}(x)\notin(-\delta,\delta)$, then
$|Df_a^k(x)|\geq\delta\e^{c_0k}$.

\end{enumerate}
While the orbit goes through a \emph{free period} its iterates are always away from the critical region which means that the above estimates apply and it experiences an exponential growth of the derivative. However, it is inevitable that the orbit of
almost every $x\in[-1,1]$ makes a \emph{return} to the critical
region. We say that $n\in\N$ is a \emph{return time} of the orbit of $x$ if $f_a^n(x)\in(-\delta,\delta)$. Every free period of $x$ ends
with a \emph{free return} to the critical region. We say that the
return has  \emph{depth} $m\in\N$ if  $f_a^n(x)\in I_{\pm m}$. Once in the critical region, the orbit of $x$ initiates a binding with the critical point.

\subsection{Bound period definition and properties}
\label{subsec:bound-period-estimates} Let $\beta=2\alpha$. For every
$x\in(-\delta,\delta)$ define $p(x)$ to be the largest integer $p$
such that
$|f_a^k(x)-f_a^k(0)|<\e^{-\beta k}$,
$\forall k<p$. For every $|m|\geq \Delta$ we define $p_m=\min_{x\in \mathscr U_m} p(x)$. The orbit of $x\in I_m^+$ is said to be \emph{bound} to the critical point during
the period $0\leq k<p_m$. The bound period $p_m$ of the points $x\in I_m$, for each $|m|\leq \Delta$ satisfies the following properties:

\begin{enumerate}

\item \label{item:bound-period-bounds-1d}$\frac{1}{2}|m|\leq p_m\leq3|m|$;

\item \label{item:bound-period-distortion} there exists $B>0$ such that $\frac 1 B\leq\frac{Df_a^k(f_a(x))}{Df_a^k(f_a(0))}\leq B$ for all $k=1,\ldots, p_m-1$

\item \label{item:bound-period-derivative-1d} $|Df_a^{p_m}(x)|\geq
\e^{(1-4\beta)|m|}$.

\end{enumerate}

The bound period plays a prominent role in the proof of the Benedicks-Carleson
Theorem. Roughly speaking, 
 while the orbit of the critical point is outside the critical region we have expansion (see Subsection
\ref{subsec:free-period-estimates}); when it returns we have a
serious setback in the expansion but then, by continuity, the orbit
repeats its early history regaining expansion on account of
\eqref{eq:EG}. To arrange for the exponential growth of the
derivative along the critical orbit \eqref{eq:EG} one has to
guarantee that the losses at the returns are not too drastic; hence,
by parameter elimination, the basic assumption condition
\eqref{eq:BA} is imposed. The argument is mounted in a very
intricate induction scheme that guarantees both the conditions for
the parameters that survive the exclusions.

\subsection{Spatial and parameter resemblances}
\label{subsec:space-parameter-derivatives}
One of the keys of the parameter exclusion argument is that one can import properties observed in the ambient space $[-1,1]$ to the set of parameters. The main tool to achieve that is the fact that, as long as we have exponential growth of the spatial derivative along the critical orbit, spatial derivatives are close to parameter derivatives. To be more specific let $\xi_n(a):=f_a^n(0)$ and $\xi'_n(a):=\frac{d(\xi_n(a))}{da}$. Then given $2/3<c<\log2$, there exists $N_0$ such that, for every $n\geq N_0$, if $Df_a^j(1)\geq 3^j$ for all $j=1,\ldots N_0$ and $Df_a^j(1)\geq \e^{cj}$ for all $j=1,\ldots,n-1$ then
\[
1/A\leq\frac{\xi'_n(a)}{Df_a^{n-1}(1)}\leq A,
\]
where $A=8$. 

Another important issue regards the bound periods whose definition in Subsection~\ref{subsec:bound-period-estimates} clearly depends on the parameter $a$. So here we will express that by writing $p_{m,a}$ to record this fact. Now, for a parameter interval $\omega$ such that $\xi_n(\omega)\subset I_m$ for some $|m|\geq \Delta$ we define
\[
p(\omega,m)=\min_{a\in\omega} p_{m,a}.
\]
It is possible to show that if $\xi_n(\omega)\subset I_m$ for some $|m|\geq \Delta$, then the properties \eqref{item:bound-period-bounds-1d} to \eqref{item:bound-period-derivative-1d} of Subsection~\ref{subsec:bound-period-estimates} hold for $p(\omega,m)$ in the place of $p_m$ and all $a\in \omega$.

\subsection{Construction of the parameter set}
\label{subsec:parameter-set}

Let $\zeta_2$ be a repelling periodic point of $f_2$, of prime period $q\in\N$ and distinct from the critical orbit. Consider its hyperbolic continuation $\zeta_a$ for $a$ sufficiently close to the parameter value 2. Our goal is to show that there exists a positive Lebesgue measure set of parameters $\Omega_\infty$ satisfying  \eqref{eq:EG}, \eqref{eq:BA} and such that condition \eqref{eq:dens} holds for $\zeta_a$, for all $a\in\Omega_\infty$. We will achieve this by imposing some sort of basic assumption \eqref{eq:BA} with respect to the periodic point $\zeta_a$. We call it \emph{periodic assumption} and basically we will require that the critical orbit does not go near the periodic orbit $\zeta_a$ too quickly. Namely, let $\gamma=\min_{j=0,\ldots,q-1}|-1-f_2^j(\zeta_2)|$, $N_1\in\N$ be such that for all $n\geq N_1$ we have $\e^{-\alpha n}<\gamma/4$ and consider
\begin{equation}
\begin{split}\min_{j=0,\ldots,q-1}\left|\xi_n(a)-f_a^j(\zeta_a)\right|&\geq \gamma/4, \mbox{ for all $n=1,\ldots,N_1$ and }\nonumber \\ \min_{j=0,\ldots,q-1}\left|\xi_n(a)-f_a^j(\zeta_a)\right|&\geq \e^{-\alpha n}, \mbox{ for all $n> N_1$.}\end{split}\tag{P\!A}\label{eq:PA}
\end{equation}

For what follows we need to introduce finite time versions of the conditions \eqref{eq:EG}, \eqref{eq:BA} and \eqref{eq:PA}, which we will denote by ($\ref{eq:EG}_n$), ($\ref{eq:BA}_n$) and ($\ref{eq:PA}_n$), respectively: these are defined in exactly the same way as the original conditions, except that they hold only up to time $n\in\N$, instead of for all the integers.
 
Next, we give the procedure of parameter exclusion of Benedicks-Carleson as in \cite{M93} with some changes in order guarantee that at the end \eqref{eq:PA} holds.

Related to the partition $\mathcal P$ of the phase space $I=[-1,1]$, we will define inductively a sequence of partitions $\mathcal{P}_0,\mathcal{P}_1,\ldots$ in the space of parameters in order to obtain bounded distortion of $\xi_n$ in each element $\omega\in\mathcal P_n$. Notice that the sets $$\Omega_n=\bigcup_{\omega\in\mathcal P_n}\omega$$ of parameters, which satisfy all the discussed conditions up to time $n$, form a decreasing sequence of parameter sets whose limit $\Omega_\infty$ will have positive Lebesgue measure.
For each $\omega\in\mathcal P_n$, we will also define inductively the
sets $R_n(\omega)=\left\{z_1,\ldots,z_{\gamma(n)}\right\}$, which correspond to the return times of $\omega\in \mathcal{P}_n$ up to $n$
and a set
$Q_n(\omega)=\left\{(m_1,k_1),\ldots,(m_{\gamma(n)},k_{\gamma(n)})\right\}$,
which records the indices of the intervals such that
$f_a^{z_i}(\omega)\subset I_{m_i,k_i}^{+}$,
$i=1,\ldots,z_{\gamma(n)}$.

To begin the construction, we proceed as in the original argument by choosing the constants appearing in \eqref{eq:EG} and \eqref{eq:BA}, namely we take some $c$ very close to $\log2$,  a small $\alpha<0.001$ and $\beta=2\alpha$. 

Let $N_0$ be as is subsection~\ref{subsec:space-parameter-derivatives}. Let $N_1$ be as in \eqref{eq:PA}.

For each $i=0, 1,\ldots,q-1$, let $V_i$ be a sufficiently small interval centred at $f_2^i(\zeta_2)$ where $f_2^q$ behaves like the linear map $x\mapsto \sigma_2 x$, where $\sigma_2:=(f_2^q)'(\zeta_2)$. Moreover assume that $\bigcup_{j=1}^q f_2^j(V_i)\cap \mathscr U_{\Delta}=\emptyset$ and $\bigcup_{i=1}^q f_2^i(V)\cap [-1,-1+\gamma/2]=\emptyset$. Suppose that all $V_i$'s are disjoint and have the same length that we denote by $|V|$. Let $N_2\in \N$ be such that for all $n\geq N_2$ we have 
\begin{equation}
\label{eq:N2-def}
9AC\delta^{-1}\e^{-\alpha n/2}<|V|.
\end{equation}

Observe that since for each $\omega\in \mathcal P_n$ condition $(EG_n)$ holds, then by the relation between spatial and parameter derivatives given in  subsection~\ref{subsec:space-parameter-derivatives}, we have that $\xi_n$ expands exponentially fast which gives 
an upper bound for the size of $\omega\in \mathcal P_n$, namely, $|\omega|\leq \mbox{const}. \e^{-cn}$. Besides, since the size of $\rho(\omega)=\{\zeta_a:\; a\in\omega\}$ is proportional to the size of $\omega$, an upper bound like the one just above applies to $|\rho(\omega)|$ with a different constant. Since $\alpha$ is much smaller than $c$, we let $N_3\in\N$ be such that for all $n\geq N_3$ we have that for all $\omega\in\mathcal P_n$, the following estimate holds: 
\begin{equation}
\label{eq:N3-def}
|\rho(\omega)|<\e^{-\alpha n}.
\end{equation}

Using the fact that for $a=2$, we have that the critical point hits the repelling fixed point $-1$, which means that $f^{j}(0)=-1$ and $|Df_2^j(1)|=4^j$, for all $j\geq 2$, we can choose $a_0$ sufficiently close to $2$ and $N>\max\{N_0,N_1, N_2\}$ so that all estimates in the original argument hold (in particular, the growth of the derivative along the critical orbit  so that spatial and parameter derivatives are comparable, up to time $N$) and moreover
\begin{equation}
\label{eq:parameter-choice}
\xi_j([a_0,2])<-1+\gamma/4, \quad \mbox{for all $2\leq j\leq N$.}
\end{equation}
Then we choose $\Delta$ large enough so that all estimates appearing throughout the original procedure hold and also such that
\begin{equation}
\label{eq:Delta-choice}
\left|f_a^j(\e^{-\Delta})-\xi_j(a)\right|<\gamma/4, \quad \mbox{for all $1\leq j\leq N_1$ and all $a\in[a_0,2]$.}
\end{equation}

Let $\Omega_0=[a_0,2]$ be the base of our construction. By our choice of $a_0$ we may assume that $\Omega_j=\Omega_0$ and $\mathcal P_j=\{\Omega_0\}$, for all $j=0,\ldots, N$. Moreover, set $R_j(\Omega_0)=Q_j(\Omega_0)=\emptyset$, for all $j=0,\ldots,N$.

Assume that $\mathcal{P}_{n-1}$ is defined as well as $R_{n-1},\,Q_{n-1}$,
on each element of $\mathcal{P}_{n-1}$. We fix an interval
$\omega\in\mathcal{P}_{n-1}$. We have three possible situations:

\begin{enumerate}

\item If $R_{n-1}(\omega)\neq\emptyset$ and
 $n<z_{\gamma(n-1)}+p(m_{\gamma(n-1)})$ then we say that $n$ is a
\emph{bound time} for $\omega$, put $\omega\in\mathcal{P}_n$ and set
$R_n(\omega)=R_{n-1}(\omega)$, $Q_{n}(\omega)=Q_{n-1}(\omega)$.

\item If $R_{n-1}(\omega)=\emptyset$ or $n\geq z_{\gamma(n-1)}+
p(m_{\gamma(n-1)})$, and $\xi_n(\omega)\cap \mathscr U_\Delta\subset
I_{\Delta,1}\cup I_{-\Delta,1}$, then we say that $n$ is a
\emph{free time} for $\omega$. Consider the intervals $J_{n,\omega}=\{\zeta_a:\, a\in\omega\}$ and its $\e^{-\alpha n}$-neighbourhood $J_{n,\omega}^+=\cup_{a\in\omega}(\zeta_a-\e^{-\alpha n}, \zeta_a+\e^{-\alpha n})$. Now, we have two possibilities. Either  $\xi_n(\omega)\cap J_{n,\omega}^+=\emptyset$, in which case we put $\omega\in\mathcal{P}_n$ and set
$R_n(\omega)=R_{n-1}(\omega)$, $Q_{n}(\omega)=Q_{n-1}(\omega)$; or $\xi_n(\omega)\cap J_{n,\omega}^+\neq\emptyset$, in which case we let $\omega_1$ and $\omega_2$ be the possible nonempty connected components of $\omega\setminus \xi_n^{-1}(J_{n,\omega}^+)$. For each $i=1,2$, if $\omega_i$  is such that $|\xi_n(\omega_i)|>2\e^{-\alpha n}$, then we put $\omega_i\in\mathcal{P}_n$ and set
$R_n(\omega_i)=R_{n-1}(\omega_i)$, $Q_{n}(\omega_i)=Q_{n-1}(\omega_i)$; otherwise we just exclude $\omega_i$ as well.

\item If the above two conditions do not hold we say that
$\omega$ has a \emph{free return situation}\index{free return
situation} at time $n$. We have to consider two cases:

\begin{enumerate}

\item $\xi_n(\omega)$ does not cover completely an interval $I_{m,k}$, with $|m|\geq \Delta$ and $k=1,\ldots,m^2$.
  Because $\xi_n$ is continuous and $\omega$ is an interval, $\xi_n(\omega)$ is also an interval and thus is contained
in some $I_{m,k}^+$, for a certain $|m|\geq \Delta$ and
$k=1,\ldots,m^2$, which is called the \emph{host interval} of the
return. We say that $n$ is an \emph{inessential return
time}\index{inessential return time} for $\omega$, put
$\omega\in\mathcal{P}_n$ and set
$R_n(\omega)=R_{n-1}(\omega)\cup\{n\}$,
$Q_{n}(\omega)=Q_{n-1}(\omega)\cup\{(m,k)\}$.

\item $\xi_n(\omega)$ contains at least an interval $I_{m,k}$, with $|m|\geq \Delta$ and $k=1,\ldots,m^2$, in which
case we say that $\omega$ has an \emph{essential return
situation}\index{essential return situation} at time $n$. Then we
consider the sets
\begin{align*}
  &\omega_{m,k}'=\xi_n^{-1}(I_{m,k})\cap\omega \quad\mbox{ for $|m|\geq\Delta$}\\
  &\omega_{+}'=\xi_n^{-1}([\delta,1])\cap\omega \\
  &\omega_{-}'=\xi_n^{-1}([-1,-\delta])\cap\omega
\end{align*}
and if we denote by $\mathcal{A}$ the set of indices $(m,k)$ such
that $\omega_{m,k}'\neq\emptyset$ we have
\begin{equation}
  \label{eq:omega-decomp}
  \omega\setminus\left\{\xi_n^{-1}(0)\right\}=
  \bigcup_{(m,k)\in\mathcal{A}}\omega_{m,k}'.
\end{equation}

By induction, $\xi_n|_\omega$ is a diffeomorphism and
then each $\omega_{m,k}'$ is an interval. Moreover
$\xi_n(\omega_{m,k}')$ covers the whole of $I_{m,k}$ except possibly
for the two end intervals. When $\xi_n(\omega_{m,k}')$ does not
cover $I_{m,k}$ entirely, we join it with its adjacent interval in
\eqref{eq:omega-decomp}. We also proceed likewise when
$\xi_n(\omega_+')$ does not cover $I_{\Delta-1,(\Delta-1)^2}$, or
 $\xi_n(\omega_-')$ does not contain the whole interval
 $I_{1-\Delta,(\Delta-1)^2}$. In
 this way we get a new decomposition of $\omega\setminus\left\{\xi_n^{-1}(0)\right\}$
into intervals $\omega_{m,k}$ such that
\[
I_{m,k}\subset \xi_n(\omega_{m,k})\subset I_{m,k}^+,
\]
when $|m|\geq\Delta$.

We define $\mathcal{P}_n$, by putting $\omega_{m,k}\in
\mathcal{P}_n$ for all indices $(m,k)$ such that
$\omega_{m,k}\neq\emptyset$, with $\Delta\leq |m|\leq \alpha n$, which results in
a refinement of $\mathcal{P}_{n-1}$ at $\omega$. Note that the sets $\omega_{m,k}$ such that $|m|>\alpha n$, for which $\xi_n(\omega_{m,k})\subset \mathscr U_{\alpha n}$ are excluded. We set
$R_n(\omega_{m,k})=R_{n-1}(\omega)\cup\{n\}$ and $n$ is called an
\emph{essential return time}\index{essential return time} for the surviving
$\omega_{m,k}$. The interval $I_{m,k}^+$ is called the \emph{host
interval}\index{host interval} of $\omega_{m,k}$ and
$Q_n(\omega_{m,k})= Q_n(\omega)\cup\{(m,k)\}$.

In the case when the set $\omega_{+}$ is not empty we say that $n$
is an \emph{escape time}\index{escape time} or \emph{escape
situation}\index{escape situation} for $\omega_{+}$ and
$R_n(\omega_{+})=R_{n-1}(\omega)$,
$Q_{n}(\omega_{+})=Q_{n-1}(\omega)$. We proceed likewise for
$\omega_{-}$. We also refer to $\omega_{+}$ or $\omega_{-}$ as
\emph{escaping components}. Note that the points in escaping
components are in free period.

\end{enumerate}

\item Before moving on to step $n+1$ there are still possible exclusions before $\mathcal P_n$ is complete. Consider an interval $\omega$ that has already been put into $\mathcal P_n$, so far. We look at $R_n(\omega)=\{z_1, \ldots, z_{\gamma(n)}\}$ and consider the sequence $p_1,\ldots, p_{\gamma(n)}$ of the lengths of the bound periods associated to each respective return. If $$\sum_{i=1}^{\gamma(n)}p_i\leq \alpha n,$$ we keep $\omega$ in $\mathcal P_n$, otherwise we remove it. This rule will enforce all elements of $\mathcal P_n$ to satisfy the so called \emph{Free Assumption}, up to time $n$, which is denoted by $FA_n$. Basically $FA_n$ assures that the time spent by the orbit of the critical points in free periods makes up a very large portion of the whole time, namely, larger than $(1-\alpha)n$.

\end{enumerate}

\begin{remark}
Note that the only difference between the original procedure of Benedicks-Carleson presented in \cite{M93} and the one we present here is in step (2) where we make exclusions when the critical orbit gets too close to the periodic orbit.
\end{remark}

\subsection{Bounded distortion}
\label{subsec:bounded-distortion} The sequence of partitions
described above is designed so that we have bounded distortion in
each element of the partition $\P_{n-1}$. To be more
precise, consider $\omega\in\P_{n-1}$. There exists a constant $C$
independent of $\omega$, $n$ such that for
every $a,b\in\omega$,
\begin{equation}
\label{eq:bounded-distortion}
 \frac{|Df_a^{n-1}(1)|}{|Df_b^{n-1}(y)|}\leq C\qquad \text{and}\qquad \frac{|\xi_n'(a)|}{|\xi_n'(b)|}\leq C
\end{equation}
See \cite[Lemma~5]{BC85} or \cite[Proposition~5.4]{M93} for a proof.

\subsection{Growth of returning and escaping components}
\label{subsec:growth-components} Let $t$ be a return time
for $\omega\in\P_{t}$, with $\xi_t(\omega)\subset 3
I_{m,k}$ for some $\Delta\leq m\leq \alpha t$ and $1\leq k\leq m^2$ and let $p$ denote the corresponding bound period. If $t$ is the last return time before $n$ and $n$ is a free time or a return situation for $\omega$ then
\begin{equation}
\label{eq:free-after-return} \left|\xi_n(\omega)\right|\geq \e^{c_0\varrho-\Delta}\e^{(1-5\beta)|m|}
        \left|\xi_t(\omega)\right|\mbox{ and if $t$ is essential then }\left|\xi_n(\omega)\right|\geq
\e^{c_0\varrho-\Delta}\e^{-5\beta|m|},
\end{equation}
where $\varrho=n-(t+p)$.
If $n$ is the
next free return time  for $\omega$ (either essential or
inessential) then
\begin{equation}
\label{eq:return-after-return} \left|\xi_n(\omega)\right|\geq \e^{c_0\varrho}\e^{(1-5\beta)|m|}
        \left|\xi_t(\omega)\right|\mbox{ and if $t$ is essential then }\left|\xi_n(\omega)\right|\geq
\e^{c_0\varrho}\e^{-5\beta|m|},
\end{equation}
where $\varrho=n-(t+p)$. See \cite[Lemma~5.2]{M93}.

Suppose that $\omega\in\mathcal{P}_t$ is an escape component. Then,
in the next return situation for $\omega$, at time $n$, we have
that
\begin{equation*}
\label{eq:return-after-escape} \left|\xi_n(\omega)\right|\geq
\e^{-\beta\Delta}.
\end{equation*}
See \cite[Lemma~5.1]{M93}.

\subsection{Estimates on the exclusions}
In the original procedure by Benedicks-Carleson the exclusions happen on account of rules (3b) and (4) above. Regarding the exclusions by applying (3b) at step $n$ we have the following estimate:
\begin{equation}
\label{eq:excluded-set-estimate1}
|\Omega_{n-1}\setminus\Omega_n'|\leq \e^{-\epsilon n}|\Omega_{n-1}|,
\end{equation}
 where $\epsilon>0$ and $\Omega_n'$ are the parameters of $\Omega_{n-1}$ that survive the exclusions forced by (3b).
Regarding the exclusions because of rule (4) we have that:
\begin{equation*}
\label{eq:excluded-set-estimate2}
|\Omega_{n}'\setminus \Omega_n|\leq \e^{-\epsilon n}|\Omega_0|,
\end{equation*}
 where $\epsilon>0$.
In the original argument, the most complicated to estimate are the exclusions resulting from applying rule (4) in order to guarantee the free assumption (FA). These are dealt with a large deviation argument for which an estimation on the probability of very deep returns is needed. However, the new exclusions we introduce here in rule (2) are more like the ones operated on account of (3b). In order to get estimate \eqref{eq:excluded-set-estimate1}, one realises first that by choice of $\beta$ if $n$ is a bound period time for $\omega$ then $\xi_n(\omega)$ is clearly outside $\mathscr U_{[\alpha n]}$. Moreover, if $\xi_n(\omega)$  hits $\mathscr U_{[\alpha n]}$ then it has already achieved large scale, meaning that the size of  $\xi_n(\omega)$ is at least $\e^{-\alpha n/2}$. In fact, using \eqref{eq:free-after-return} and \eqref{eq:return-after-return} one can show that (see \cite[Lemma~5.3]{M93}) if $n$ is either a free time or a return situation for $\omega\in\mathcal P_{n-1}$, we have that
\begin{equation}
\label{eq:large-scale} \left|\xi_n(\omega)\right|\geq
\e^{-\alpha n/2}.
\end{equation}
Once large scale is achieved, using bounded distortion one gets that the exclusions produced on $\Omega_{n-1}$, by applying (3b), are at most proportional to $\frac{\e^{-\alpha n}}{\e^{-\alpha n/2}}=\e^{-\alpha n/2}$. Hence, estimate \eqref{eq:excluded-set-estimate1} follows with $\epsilon=\alpha /2$.   

With this new procedure, estimate \eqref{eq:large-scale} still holds during free times, then whenever we make exclusions by applying rule (2), we have already reached large scale and the same argument can be used to obtain an estimate like  \eqref{eq:excluded-set-estimate1} for the new exclusions we incorporated in rule (2). 
Thus, basically, we have to check the following facts in order to conclude that our modifications do not tamper much with the original procedure, and the new exclusions still allow to obtain a positive Lebesgue measure set of parameters $\Omega_\infty$ satisfying simultaneously conditions (EG), (BA), (FA) and (P\!A):
\begin{enumerate}

\item if $n$ is a bound time for $\omega\in\mathcal P_{n-1}$ then $\xi_n(\omega)$ is away $J_{n,\omega}^+$.

\item if an exclusion occurs at time $n$ by applying rule (2) then the remaining connected components of $\Omega_n$ still achieve large scale (meaning that estimate \eqref{eq:large-scale} holds) before new exclusions may occur. 

\end{enumerate}
     
\begin{proposition}
\label{prop:bound}
Suppose that $\omega\in\mathcal P_{n-1}$ and $(P\!A_{n-1})$ holds for $\omega$. Assume also that there is $s\leq n-1$ such that $\xi_s(\omega)\subset \mathscr U_{\Delta}$, $p_s$ is the bound period associated to this return at time $s$ and $n<s+p_s$. Then condition $(PA_n)$ holds for $\omega$.
\end{proposition}
\begin{proof}
Suppose first that $n-s\leq N_1$. By \eqref{eq:Delta-choice} we have $|\xi_n(a)-\xi_{n-s}(a)|<\gamma/4$. By \eqref{eq:parameter-choice}, it follows that $\xi_{n-s}(a)<-1+\gamma/4$. Hence, putting it together, for all $a\in\omega$, we have $\xi_{n}(a)<-1+\gamma/2$, which means that $(PA_n)$ holds for $\omega$ in this case.

Now assume that $n-s>N_1$. Since $(P\!A_{n-1})$ holds for $\omega$, which implies that $(P\!A_{n-s})$ also holds, then $$\min_{j=0,\ldots,q-1}|\xi_{n-s}(a)-f_a^j(\zeta_a)|\geq \e^{-\alpha (n-s)}.$$  By the binding condition we also have 
$
|\xi_n(a)-\xi_{n-s}(a)|<\e^{-\beta (n-s)}. $ Thus, if $s>N$ is sufficiently large then
\begin{align*}
\min_{j=0,\ldots,q-1}|\xi_{n}(a)-f_a^j(\zeta_a)|&>\e^{-\alpha (n-s)}-\e^{-\beta (n-s)}=\e^{-\alpha(n-s)}\big(1-\e^{(\alpha-\beta)(n-s)}\big)\\
&>\e^{-\alpha(n-s)}\big(1-\e^{(\alpha-\beta)}\big)=\e^{-\alpha n} \e^{\alpha s}\big(1-\e^{-\alpha}\big)\\
&>\e^{-\alpha n},
\end{align*}
which means that $(P\!A_n)$ holds for $\omega$ also in this case.
\end{proof}

As a consequence of Proposition~\ref{prop:bound} and the rule (2) of the procedure we have that if $(PA_{n-1})$ holds for all $\omega\in \mathcal P_{n-1}$ then $(P\!A_{n})$ holds for all $\omega\in \mathcal P_{n}$. This is because for any $\omega\in \mathcal P_{n-1}$ either $n$ is a bound period, in which case Proposition~\ref{prop:bound} gives the conclusion, or $n$ is in a free period. When $n$ is in free period either $\xi_n(\omega)\cap J_{n,\omega}^+=\emptyset$, in which case the conclusion is clear, or else we have to make exclusions according to rule (2) so that for the reminders of $\omega$ that go into $\mathcal P_n$ it is also clear that  $(P\!A_{n})$ holds. This means that condition \eqref{eq:PA} eventually holds for all parameters in $\Omega_\infty$.

Finally, we have to check that the exclusions on account of the changes we included in rule (2), still allow to achieve large scale before new exclusions occur. Recall that this is crucial for the estimates on the excluded sets to hold and it means that we have to check that condition \eqref{eq:large-scale} holds when new exclusions are about to happen. The problem that could arise would be that when make exclusions using rule (2), some small leftovers could possibly not have had time to reach large scale. We will see that this does not happen. This is essentially because   when a cut to $\omega$ occurs because $\xi_n(\omega)$ is too close to the periodic points $\zeta_a$, the possibly remaining small bits (whose size is at least $2\e^{-\alpha n}$, by construction) are so close to the repelling periodic points that they will shadow them for long enough time to grow up to reach the size $\e^{-\alpha n/2}$.   

\begin{proposition}
\label{prop:growth}
Let $\omega\in\mathcal P_{n-1}$ and assume that $\xi_n(\omega)\cap J_{n,\omega}^+\neq\emptyset$ which means that an exclusion is made according to rule (2). Let $\omega^*\in\mathcal P_n$ be a surviving connected component of $\omega$. Let $m$ be the such that $m-1\geq n$ is the last time $\omega^*\in\mathcal P_{m-1}$, either because $\xi_m(\omega^*)\cap J_{n,\omega}^+\neq\emptyset$, which means that new exclusions must be made, or $m$ is a returning time situation, which means that a refinement, with possibly some exclusions, will occur. Then $$|\xi_m(\omega^*)|\geq \e^{-\alpha n/2}>\e^{-\alpha m/2}.$$ 
\end{proposition}

\begin{proof}
Let $V$ denote the neighbourhood $V_i$ such that $\xi_n(\omega)\cap V_i\neq \emptyset$. Since $n>N$ is very large and $a_0$ is very close to 2, we may assume that the neighbourhood $V$ appearing in \eqref{eq:N2-def} is the same for every $a\in\omega^*$ and it has the same properties when iteration by $f_2$ is replaced by $f_a$. 

Let $x,y\in\xi_n(\omega^*)$ be respectively the closest point to and the farthest away point from  $\rho(\omega^*)$. Note that by \eqref{eq:N3-def}, we have $\e^{-\alpha n}<\dist(x,\rho(\omega^*))<2\e^{-\alpha n}$ and, by construction, $\dist(y,\rho(\omega^*))>3\e^{-\alpha n}$. This gives that $\dist(x,y)>\dist(y,\rho(\omega^*))/3$. For definiteness let $x<y$ and $a$ be such that $\zeta_a$ is the closest point to $x$ in $\rho(\omega^*)$. Also, let $\sigma:=\sigma_a=|(f_a^q)'(\zeta_a)|$. Now, $f_a^q:V\to f_a^q(V)$ behaves just like the linear map $x\mapsto \sigma x$ with the origin coinciding with $\zeta_a$. This means that for $n$ sufficiently large we have to wait some time before the interval $[x,y]$ leaves $V$ by iteration by $f_a^q$.  Let $j$ be the first time that $f_a^{jq}(y)$ is outside of $V$. Then $\dist(\zeta_a,f_a^{jq}(y))> |V|/3$. The size of $|f_a^{jq}([x,y])|$ is then approximately $\dist(f_a^{jq}(x),f_a^{jq}(y))\asymp\sigma^j\dist(x,y)>\sigma^j\dist(y,\zeta_a)/3>|V|/9$. Now, since spatial and parameter derivatives are very close to each other then $(f_a^{jq})'(x)$ is approximately $\xi_{jq}'(a)$. This together with the bounded distortion of the parameter derivatives means that, the size of $\xi_{n+jq}(\omega^*)$ will be approximately $|V|/(9CA)$, where $C$ comes from bounded distortion and $A$ from the relation between spatial and parameter derivatives.  

Note that for $n<t<n+jq$ we have that $\xi_t(\omega^*)\subset\cup_{i=0}^{q-1}f_a^i(V)$ which by choice of $V$ means that you cannot have exclusions nor refinements of the partition since you are clearly away from $J_{n,\omega^*}^+$ and $\mathscr U_\Delta$. Now that we have seen that $m>n+jq$, recall that between $n$ and $m$ we are still in free period which means that as in \eqref{eq:free-after-return} we have that $|\xi_m(\omega^*)|\geq \delta \e^{c_0(m-(n+jq))}|\xi_{n+jq}(\omega^*)|$. Finally, by \eqref{eq:N2-def}, we have that
$$
|\xi_m(\omega^*)|\geq  \delta |V|/(9CA)>\e^{-\alpha n/2}>\e^{-\alpha m/2},
$$ 
as required.
\end{proof}

\bibliographystyle{novostyle}

\bibliography{ExtremalIndexPoisson}

\end{document}